\documentclass[11pt]{amsart}

\usepackage{amsmath, amsthm, amssymb}

\usepackage{lmodern, enumerate}
\usepackage{palatino}                       % text font
\usepackage[bigdelims,vvarbb]{newpxmath}    % math font
\usepackage[scaled=0.95]{inconsolata}       % Sans font 
\usepackage[T1]{fontenc}                    % font encoding 

\usepackage{comment}

\usepackage{tikz-cd}

\usepackage[abs]{overpic}		  % overpic automatically loads graphicx

\usepackage{xcolor}
\definecolor{indigo}{rgb}{0.29, 0.0, 0.51}  % custom colors
\usepackage[colorlinks, urlcolor=indigo, linkcolor=indigo, citecolor=indigo]{hyperref}

\usepackage[hcentering, top = 1.5in, total={5.9in, 8.2in}]{geometry}  % 1.4 vertical margin; 1.3 horizontal margin

\theoremstyle{plain}
\newtheorem{theorem}{Theorem}

\newtheorem{proposition}[theorem]{Proposition}
\newtheorem{lemma}[theorem]{Lemma}

\newtheorem{mainthm}{Theorem}

\theoremstyle{definition}
\newtheorem{definition}[theorem]{Definition}

\theoremstyle{remark}
\newtheorem{remark}[theorem]{Remark}

\newtheorem{observation}[theorem]{Observation}

\numberwithin{theorem}{section}

\newcommand{\R}{\mathbb{R}}           % the real numbers
\newcommand{\Z}{\mathbb{Z}}           % the integers
\newcommand{\C}{\mathbb{C}}           % the complex numbers
\newcommand{\N}{\mathbb{N}}           % the natural numbers
\DeclareMathOperator{\bd}{\partial}   % boundary
\makeatletter
\newcommand*\bigcdot{\mathpalette\bigcdot@{0.6}}
\newcommand*\bigcdot@[2]{\mathbin{\vcenter{\hbox{\scalebox{#2}{$\m@th#1\bullet$}}}}}
\makeatother

\DeclareMathOperator{\reg}{\text{reg}}
\DeclareMathOperator{\sing}{\text{sing}}
\DeclareMathOperator{\exot}{\text{exot}}
\DeclareMathOperator{\im}{\text{im}}
\DeclareMathOperator{\Ima}{\text{Im}}
\DeclareMathOperator{\Rea}{\text{Re}}
\newcommand{\wind}{\operatorname{wind}}
\newcommand{\Mod}{\operatorname{Mod}}
\newcommand{\SMod}{\operatorname{SMod}}
\newcommand{\Stein}{\operatorname{Stein}}
\newcommand{\ind}{\operatorname{ind}}

\newcommand{\crit}{\operatorname{crit}}

\newcommand{\bC}{\mathbb{C}}
\newcommand{\bD}{\mathbb{D}}

\newcommand{\bR}{\mathbb{R}}

\newcommand{\cF}{\mathcal{F}}

\newcommand{\cH}{\mathcal{H}}

\newcommand{\cM}{\mathcal{M}}

\newcommand{\cS}{\mathcal{S}}
\newcommand{\cT}{\mathcal{T}}

\usepackage{thmtools} 
\usepackage{thm-restate}

\begin{document}

% title
\title{Spinal open books and symplectic fillings with exotic fibers} 

\author{Hyunki Min}

\author{Agniva Roy}

\author{Luya Wang}

\address{Department of Mathematics \\ University of California \\ Los Angeles, CA}
\email{hkmin27@math.ucla.edu}

\address{Department of Mathematics \\ Boston College \\ Chestnut Hill, MA}
\email{agniva.roy@bc.edu}

\address{Institute for Advanced Study \\ Princeton, NJ}
\email{luyawang@ias.edu}

%\subjclass[2020]{57R17}

% abstract
\begin{abstract}
Spinal open book decompositions provide a natural generalization of open book decompositions. We show that any minimal symplectic filling of a contact $3$-manifold supported by a planar spinal open book is deformation equivalent to the complement of a positive multisection in a bordered Lefschetz fibration, which generalizes a result of Wendl \cite{Wendl_strongly}. Along the way, we give an explicit local model for a non-compactly supported ``singularity at infinity'' in a generalized version of bordered Lefschetz fibrations, given by pseudoholomorphic foliations associated to the spinal open books. This provides new tools to classify symplectic fillings of a contact $3$-manifold that is not supported by an amenable spinal open book, by studying monodromy factorizations in the newly defined spinal mapping class group. As an application, we complete the classification of strong fillings of all parabolic torus bundles, and make progress towards classifying symplectic fillings of contact $3$-manifolds supported by non-planar open books.
\end{abstract}

\maketitle
\tableofcontents

%%%%%%%%%%%%%%%%%%%%%%%%
\section{Introduction}\label{sec: intro}
%%%%%%%%%%%%%%%%%%%%%%%%

Finite-energy pseudoholomorphic foliations have played a central role in symplectic geometry for the past forty years. Using pseudoholomorphic foliations, Gromov \cite{gromov85} and McDuff \cite{mcduff_rational_ruled} initiated the study of the classification of symplectic $4$-manifolds. In contact geometry, early major applications include Eliashberg's classification of symplectic fillings of the tight $3$-sphere \cite{eli90} and Hofer's resolution of the Weinstein conjecture in several cases \cite{Hofer_Weinstein_conjecture}. These works inspired many developments later on in symplectic field theory (SFT), symplectic dynamics, and contact topology. 

Our paper focuses on the problem of classification of symplectic fillings\footnote{In this paper, we often omit \emph{strong} in front of symplectic filling when it does not cause confusion. A strong symplectic filling $(X, \omega)$ of the contact manifold $(Y,\xi)$ is such that $\partial X = Y$ and $\omega|_{\text{nbd}\partial X} = d\lambda$ for some contact form $\lambda$ of $\xi$.} of contact $3$-manifolds. For the tool of pseudoholomorphic foliations to apply, one has to restrict to the case where the leaves of the foliations are planar, due to index reasons. In the planar case, or more precisely, in the case of contact $3$-manifolds supported by planar open book decompositions, there are many past successes in the classification of their symplectic fillings, which is typically considered a very difficult problem, see \cite{Wendl_strongly, PVHM:lensFillings, Kaloti:fillings, pla_lensspace, wand2012mapping, kaloti-li} for examples. To go beyond the planar open book case, one key recent development is the technology of spinal open books initiated by Lisi, Van-Horn-Morris and Wendl in \cite{spinal_I} and \cite{spinal_II}. We discuss the current status and definition of spinal open books in \S\ref{subsec:past_planar_spinal obd} and \S\ref{subsec: spinalobd}. Our paper completes the topological understanding of a novel element in these studies: the role of ``exotic fibers'' in bordered Lefschetz fibrations that fill planar spinal open books. As a result, we are now able to give a complete description of symplectic fillings incorporating this new phenomenon. Our main theorems are stated in \S\ref{subsec:main_results}. 

\subsection{Scientific context} The main results of \cite{Wendl_strongly}, \cite{spinal_I} and \cite{spinal_II} encapsulate a general theme of classifying minimal symplectic fillings of contact $3$-manifolds. The primary idea is to find planar pages in the contact $3$-manifold, which are genus-zero pages of its open book decomposition and which can be lifted to pseudoholomorphic curves in the symplectization. One then considers the moduli space of these pseudoholomorphic curves which turn out to foliate the symplectic fillings of the contact $3$-manifold, as in the main theorems of \cite{Wendl_strongly, spinal_I, spinal_II}. In \cite{spinal_I} and \cite{spinal_II}, which generalize the results of \cite{Wendl_strongly}, the authors show that in general, the moduli spaces contain two types of degenerations, namely {\bf singular fibers} and {\bf exotic fibers}. The exotic fibers do not appear in the setting of planar open books of \cite{Wendl_strongly}, and the authors of \cite{spinal_I} and \cite{spinal_II} also only classify symplectic fillings in the restricted {\bf Lefschetz-amenable} setting, where exotic fibers do not appear. They show that, in this situation, all minimal strong fillings are deformation equivalent to Stein fillings. Furthermore, the pseudoholomorphic foliation gives a bordered Lefschetz fibration structure on the strong filling, which induce the planar spinal open book on the boundary. These results already give vast generalization of the relative case of a result of McDuff \cite{mcduff_rational_ruled} and unify the results of Gromov \cite{gromov85}, Eliashberg \cite{eli90}, and McDuff \cite{mcduff_contact_type}, which classified the minimal symplectic fillings of the standard $3$-sphere and the standard lens spaces. In addition, the main result of \cite{Wendl_strongly} turns the classification problem for $3$-manifolds supported by planar open books into a monodromy factorization problem of the mapping classes supported in the interior of the page. This reduction to monodromy factorization was used extensively in \cite{PVHM:lensFillings}, \cite{Kaloti:fillings}, \cite{pla_lensspace}, \cite{wand2012mapping}, and \cite{kaloti-li} to classify symplectic fillings of certain families of $3$-manifolds.

In this paper, we extend the results above to turn the problem of classifying symplectic fillings of a general planar spinal open book into a monodromy factorization problem, in a larger mapping class group that we call the {\bf spinal mapping class group} of a surface. This mapping class group in particular contains diffeomorphisms of a surface that are not supported in the interior and will be defined in \S\ref{subsec: framed_mcg}. We show that the presence of exotic fibers in minimal symplectic fillings forces the fillings to be complements of multisections in bordered Lefschetz fibrations. This phenomenon of exotic fiber can be considered as a ``singularity at infinity'' and has been largely unexplored. We expect our local model and description in terms of the spinal mapping class group to be also of interest for studying links of algebraic singularities in the spirit of \cite{Plamenevskaya_Starkston, Plamenevskaya_VHM, kaloti-li} and mirror symmetry.

%-----------------------------------------------------------------
\subsection{Past structural theorem for planar spinal open books}\label{subsec:past_planar_spinal obd}
%-----------------------------------------------------------------
Though the nomenclature and formal definition appeared later via the collaborations \cite{spinal_I,spinal_II}, the notion of spinal open books was first studied by Wendl in \cite{Wendl_strongly}, where he used them to classify the symplectic and Stein fillings of the standard $3$-torus.

Roughly speaking, a \textbf{spinal open book decomposition} of a $3$-manifold is generalization of the usual open book decomposition, where the binding solid tori are replaced by $S^1$-fibrations, called \textbf{spine}, over arbitrary compact oriented surfaces, called \textbf{vertebrae}. The complement of the spine in the $3$-manifold is called \textbf{paper}. A more precise definition of spinal open books can be found in \S\ref{subsec: spinalobd}. A spinal open book is \textbf{partially planar} if its interior contains a page of genus zero. The following is the main structural theorem for partially planar spinal open books proven in \cite{spinal_II} which we will develop further in our paper. See \S\ref{subsec:main_results} for our main results.

\begin{theorem}{\cite[Proposition 1.30]{spinal_II}}\label{thm: lvw_foliation}
  Let $(W, \omega)$ be a strong filling of a contact $3$-manifold $(M,\xi)$ supported by a partially planar spinal open book. Then there is a symplectic completion $\widehat{W}$ of $(W, \omega)$ with a compatible almost complex structure $\widehat{J}$ and smooth surjection
  $$\Pi: \widehat{W} \to \mathcal{M},$$
  where $\mathcal{M}$ is an oriented surface with cylindrical ends bijective to the connected components of the paper and admitting a partition
  $$\cM = \cM_{\reg}(\widehat{J}) \cup \cM_{\sing}(\widehat{J}) \cup \cM_{\exot}(\widehat{J}).$$
  Here,
  \begin{itemize}
    \item $u\in \Pi^{-1}(\cM_{\reg}(\widehat{J}))$ is an embedded $\widehat{J}$-holomorphic curve asymptotic to simply covered Reeb orbits;
    \item $u\in \Pi^{-1}(\cM_{\sing}(\widehat{J}))$ is a nodal $\widehat{J}$-holomorphic curve asymptotic to simply covered Reeb orbits composing of two embedded curves intersecting transversely exactly once;
    \item $u\in \Pi^{-1}(\cM_{\exot}(\widehat{J}))$ is an embedded $\widehat{J}$-holomorphic curve with one asymptotic end doubly covering a Reeb orbit and all other ends simply covering Reeb orbits.
  \end{itemize}
  Furthermore, for each vertebra $\Sigma_i$, there is a properly embedded $J$-holomorphic curve $S_i\subset \widehat{W}$ such that $S_i$ is homemomorphic to $\Sigma_i$ and 
  $$\Pi|_{S_i}: S_i \to \cM$$
  is a proper branched cover with simple branch points.
\end{theorem}

The above moduli spaces are referred to as {\bf the moduli space of regular, singular and exotic fibers}, respectively. The pseudoholomorphic foliations given by Theorem \ref{thm: lvw_foliation} in the symplectic filling naturally give constraint to its boundary spinal open book. This leads to the seemingly technical definition of \textbf{uniformity} of a spinal open book, given in \S\ref{subsec: spinalobd}. As discussed in \cite[Lemma 6.33]{spinal_II}, $\Pi|_{S_i}$ is an honest covering map if and only if $\cM_{\exot} = \varnothing$. This condition on the spinal open book is called {\bf Lefschetz-amenable}. In the Lefschetz-amenable setting, Lisi--Van-Horn-Morris--Wendl obtained the following result:

\begin{theorem}{\cite[Theorem 1.5]{spinal_II}}\label{thm: lef_amenable_LVW}
  Suppose $(M, \xi)$ is a closed contact $3$-manifold that is strongly fillable and contains a compact domain $M_0 \subset M$, possibly with boundary, on which $\xi$ is supported by a partially planar spinal open book $\pi$. Then $M =M_0$ and $\pi$ is uniform. Moreover, if $\pi$ is also Lefschetz-amenable, then any minimal strong filling of $(M, \xi)$ is symplectic deformation equivalent to a Stein filling.
\end{theorem}
Note that the uniformity conclusion in Theorem \ref{thm: lef_amenable_LVW} indeed implies that the partially planar spinal open book that one starts with is in fact planar. For this reason, we will state our results in this paper in terms of planar spinal open books.

\begin{remark}
  Notice that Theorem \ref{thm: lvw_foliation} and Theorem \ref{thm: lef_amenable_LVW} is stated with the condition of strong fillings. In fact, Theorem 1.10 in \cite{spinal_II} shows that every weak filling of a partially planar contact $3$-manifold $(M,\xi)$ that is exact on the spine is weakly symplectically deformation equivalent to a strong filling of $(M,\xi)$. Therefore, we may state Theorem \ref{thm: lvw_foliation}, Theorem \ref{thm: lef_amenable_LVW}, as well as our main theorems in terms of weak fillings exact on the spines and weak symplectic deformation equivalences. However, in our paper we are mostly interested in strong fillings up to symplectic deformation equivalence, so we will state the theorems in terms of strong fillings. 
\end{remark}

The use of pseudoholomorphic foliations to provide rigidity results in symplectic geometry and dynamics is a well-established, powerful, and fruitful tool. The proof of Theorem \ref{thm: lvw_foliation} by Lisi--Van Horn-Morris-Wendl in \cite{spinal_II} hinges upon the analysis of the SFT compactification of the moduli space of index two pseudoholomorphic foliations in a symplectic completion. Using this method to study finite-energy foliations for dynamical applications started with \cite{Hofer_Weinstein_conjecture} and many of the more general technical ingredients are developed in \cite{HWZI, HWZII, HWZ03}. Recently, this strategy has been applied in \cite{Hryniewicz_Salomao, Hryniewicz_Al_Salomao, Broken_books} with great success for dynamical questions. For more geometric topological questions, this has been a long-term development of Wendl, starting primarily in \cite{Wendl_strongly, Wendl_hierarchy} by using finite-energy foliations to understand fillings of planar open books and develop filling obstructions. It is also the main technical tool in \cite{nearby_Lag_T*T2}. The results of \cite{spinal_I,spinal_II} can be regarded as a continuation of the above mentioned projects. As previously mentioned, these ideas using foliations of pseudoholomorphic curves to classify symplectic $4$-manifolds date back to Gromov \cite{gromov85} and McDuff \cite{mcduff_rational_ruled}. In the usual game of using moduli spaces of pseudoholomorphic curves, there are mainly two concerns: compactness and transversality. The corresponding technical tools used for the setting of punctured finite-energy foliations are in \cite{Wendl_compactness} and \cite{Wendl_aut_tran}, together with the intersection theory of punctured pseudoholomorphic curves developed in \cite{Siefring_asymptotics, Siefring_intersection_theory, Hutchings_index_inequality}, building upon \cite{HWZI, HWZII, HWZ03} previously mentioned. We note that the automatic transversality criterion is satisfied for the moduli spaces of interest in the proof of Theorem \ref{thm: lvw_foliation} as a result of the curves considered being all of genus zero. To analyze the compactness problem, \cite{spinal_II} carefully analyzes all possible configurations arising as the SFT limit and obtains the classification of curves as cited in Theorem \ref{thm: lvw_foliation}. 

The series of works \cite{spinal_I, spinal_II} assembles multitudes of important techniques and geometric insights. However, due to its technical nature, the development of the spinal open book machinery, at the moment of their preprints, was ahead of its time and did not generate as many direct applications based on it. One exception is \cite{baykurvhm} finding contact $3$-manifolds with arbitrarily large Stein fillings. In contact and symplectic geometry of higher-dimensional manifolds, the technology of spinal open book has also inspired surprising applications. For example, Massot-Niederkr\"uger-Wendl found weakly but not Stein fillable manifolds \cite{Massot_Niederkruger_Wendl} using the spine removal surgery (see \S \ref{subsec:spine_removal}). Work of Bowden-Gironella-Moreno and Bowden-Gironella-Moreno-Zhou on the striking discovery of the abundance of tight non-fillable contact structures \cite{bowden_gironella_moreno_22, bowden2022tight} also used the technology of spinal open books. 

%Our paper obtains a complete understanding of symplectic fillings of planar spinal open books, introducing a new element of a type of singularity at infinity to the usual open book and Lefschetz fibration machinery. 
Via the full topological understanding of the associated pseudoholomorphic foliations with exotic fibers, we expect the technology of spinal open books and nearly Lefschetz fibrations to generate a new wave of activities. For example, nearly Lefschetz fibrations make a key contribution to the program of understanding fillings of links of sandwiched singularities by Plamenevskaya-Starkston \cite{plamenevskaya2025sandwichedsingularitiesnearlylefschetz}. In addition, in a sequel to the present article, the authors are working on classifying Stein and symplectic fillings of integer surgeries on various torus knots. Currently in the literature, a full classification of symplectic fillings for a given 3-manifold only exists for lens spaces \cite{Lisca:lensFillings,christian_li,etnyre_roy} and parabolic torus bundles; the latter done in this article combined with \cite{spinal_II}. By applying our new technology, the sequel contributes yet another large class of full classification results previously inaccessible without the spinal open book tools.

%----------------------------
\subsection{Main results} \label{subsec:main_results} 
%----------------------------
We extend the results in the previous section to strong fillings of contact structures supported by uniform, non-Lefschetz-amenable spinal open books. The main contribution of our paper is to study the mysterious exotic fibers naturally arising in the non-Lefschetz-amenable situation, give local models and use them to classify fillings. As a result, we obtain a strong fillability criterion for non-Lefschetz-amenable planar spinal open books. This enlarges the types of singularities in the literature that is usually considered in the classification problem for compact symplectic manifolds in a significant way. See Remark \ref{ref:exotic_in_literature} for other places where this type of singularity has been observed and relation to other works. In light of understanding these exotic fibers, we prove the following main result.

\begin{mainthm} \label{thm: main}
  Let $(M, \xi)$ be a contact $3$-manifold supported by a planar spinal open book and $(W,\omega)$ a minimal strong symplectic filling of $(M,\xi)$. Then, $(W,\omega)$ is symplectic deformation equivalent to the complement of a neighborhood of positive multisections in a bordered Lefschetz fibration.
\end{mainthm}

This result hinges on the explicit local model that we show in \S\ref{sec:exotic_fibers} as well as a work of Baykur-Hayano \cite{baykur2016multisections} in studying monodromy factorizations in the complement of multisections in Lefschetz fibrations. While the developments of spinal open books as well as multisection complements have been recent, the heuristics of the local models of exotic fibers can be traced back to 1980s, first in the work of \cite{rudolph2004algebraic}.

For the classification problem of interest, we define an object called {\bf positive allowable nearly Lefschetz fibration (PANLF)} in \S\ref{subsec: panlf}, which generalizes bordered Lefschetz fibrations and which we show to be the structure induced on a compact strong filling by the foliation in Theorem~\ref{thm: lvw_foliation}. In a similar way as positive allowable factorizations of open book monodromies correspond to fillings that are relatively minimal Lefschetz fibrations, we define {\bf positive admissible factorizations} of the page monodromy of a spinal open book with respect to $(B, \rho)$, where $B$ is a surface that the spinal open book is uniform with respect to, and $\rho$ is a representation of $\pi_1(B)$ into the mapping class group of the page of the spinal open book, detailed in \S\ref{subsec: factorization}. The second main result of this paper is showing that the classification problem of symplectic fillings of a planar spinal open books can be interpreted as a problem of factorizing the monodromy. However, in this case we have to consider mapping classes that are not necessarily supported in the interior, i.e., they can move boundary components. We will call the collection of these mapping classes the {\bf spinal mapping class group}, denoted as $\SMod(P)$, where $P$ denotes the page of a uniform spinal open book.

\begin{mainthm}\label{thm: monodromy_fact}
  Let $(M, \xi)$ be a contact 3-manifold supported by a planar spinal open book $\pi$ and $\mathcal{B}$ a collection of surfaces that $\pi$ is uniform with respect to. Then a minimal strong filling of $(M, \xi)$ corresponds to a positive admissible factorization of the monodromy of the spinal open book with respect to some surface $B \in \mathcal{B}$. 
    
  Furthermore, if $\pi$ is uniform with respect to $\bD^2$, then any minimal strong filling arising from a positive admissible factorization of the monodromy with respect to $\bD^2$ is deformation equivalent to a Stein filling. If $\pi$ is uniform with respect to $B$ with one vertebra component being an honest cover of $B$, then any minimal strong filling arising from a positive admissible factorization of the monodromy with respect to $B$ is deformation equivalent to an exact filling.
\end{mainthm}

\begin{remark}\label{rem: equivalence of fillings}
  The equivalence relation in Theorem \ref{thm: monodromy_fact} is intentionally vague and will be clarified in \S\ref{sec: constructions}. To be precise, we are considering both sides up to deformation equivalence. We say two factorizations are deformation equivalent, if their correspondingly constructed symplectic fillings in \S\ref{sec: constructions} are deformation equivalent. In \cite{Baykur_Hayano_Hurwitz}, it was shown that the generalized Hurwitz equivalence identifies different factorizations of the same filling up to additional equivalences for complements of multisections in Lefschetz fibrations over $S^2$. However, there could be further equivalence relations between fillings coming from positive admissible factorizations with respect to distinct bases $B_1, B_2 \in \mathcal{B}$, that we do not explore further in this paper. For the Weinstein Lefschetz fibrations, this equivalence is explored in \cite[Corollary 1.3.3]{breen2024girouxcorrespondencearbitrarydimensions}. 
\end{remark}

We prove several technical results in \S\ref{sec:exotic_fibers} and \S\ref{sec: constructions} that allow us to conclude Theorem~\ref{thm: monodromy_fact}. Firstly, Theorem~\ref{thm: lvw_foliation} tells us that the completion of any strong filling of a partially planar spinal open book consists of a partition by regular, singular and exotic fibers. We establish a local model for exotic fibers in Proposition \ref{prop:unique_local_model}, which gives us a count of exotic fibers in the filling, as well as the monodromy around an exotic fiber. In terms of spinal open books, we state the following theorem.

\begin{mainthm}\label{thm: intro-localmodel}
  Let $(W, \omega)$ be a minimal strong filling of a contact manifold $(M,\xi)$ supported by a planar spinal open book. Let $\cM$, $\cM_{\exot}$, and $\Pi|_{S_i}$ be as in the statement of Theorem~\ref{thm: lvw_foliation}. Then, $|\cM_{\exot}|$ is equal to the number of branch points of $\Pi|_{S_i}$. Furthermore, the monodromy around an exotic fiber in a compactified filling is a boundary interchange on the nearby regular fibers, given by Figure~\ref{fig:boundarytwist}, and $(W,\omega)$ has the structure of a PANLF.
\end{mainthm}

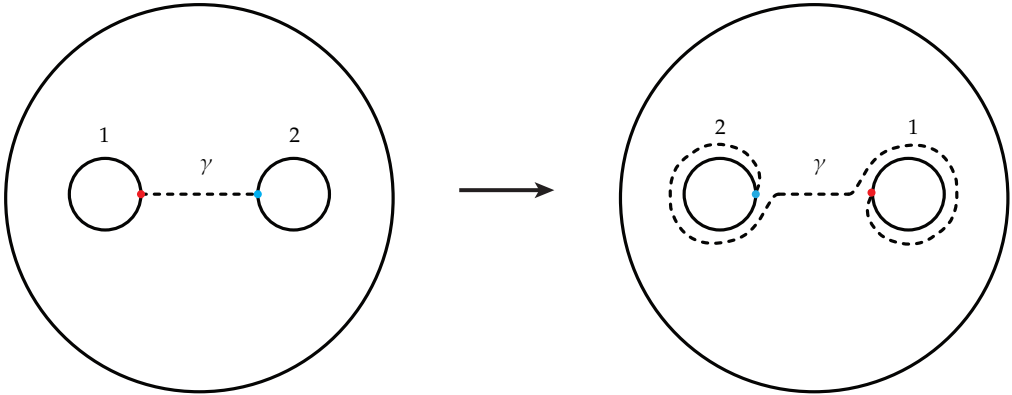
\begin{figure}[htbp]{
  \vspace{0.2cm}
  \begin{overpic}[tics=20]
  {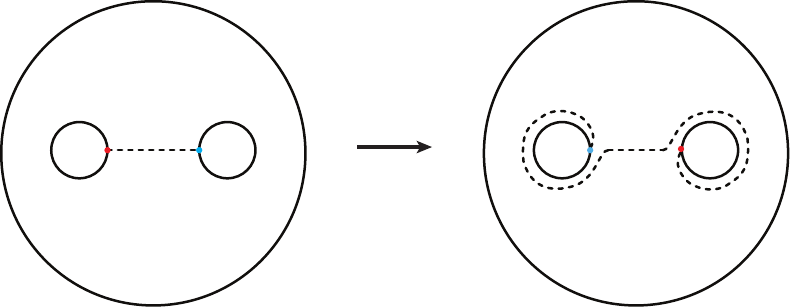}
   \put(36,95){\scriptsize$1$}
   \put(108,95){\scriptsize$2$}
   \put(74,85){\scriptsize$\gamma$}
   
   \put(269, 98){\scriptsize$2$}
   \put(342, 98){\scriptsize$1$}
   \put(306, 85){\scriptsize$\gamma$}
  \end{overpic}}
  \vspace{0.2cm}
  \caption{A local description of a \textbf{boundary interchange} $\tau_{\gamma}$, where $\gamma$ is the equatorial arc connecting the two boundary components in the interior.}
  \label{fig:boundarytwist}
\end{figure}

The more detailed definition of the boundary interchange monodromy and background on the spinal mapping class group are in \S\ref{subsec: framed_mcg}. The comparison to existing monodromy in the framed mapping class defined by \cite{baykur2016multisections} is in Remark \ref{rem:framed_mcg}.

\begin{remark} \label{ref:exotic_in_literature}
  This boundary interchange monodromy does not occur accidently. In addition to its local model being rigidly given by pseudoholomorphic foliation (Proposition \ref{prop:unique_local_model}), this singularity has also been encountered in the context of mirror symmetry. See \S8.2.11 in \cite{aganagic2023knotcategorificationmirrorsymmetry}.
\end{remark}

Secondly, the proof of Theorem~\ref{thm: monodromy_fact} involves the construction of symplectic fillings corresponding to positive admissible factorizations, which we will state now. The main idea behind the constructions is that PANLF correspond to complements of neighborhoods of positive multisections in bordered Lefschetz fibrations, building on works of Baykur-Hayano \cite{baykur2016multisections, Baykur_Hayano_Hurwitz}. 

\begin{mainthm}\label{thm: panlf_construct}
  Let $(M, \xi)$ be a contact 3-manifold supported by a planar spinal open book with page $P$ and monodromy $\phi$. Suppose the spinal open book is uniform with respect to $B$, and let $\rho: \pi_1(B) \to \Mod(P)$ be a monodromy representation. Then any positive admissible factorization of $\phi$ with respect to $(B, \rho)$ is Hurwitz equivalent to one induced by the nearly Lefschetz fibration structure of a minimal strong filling of $(M, \xi)$. 
\end{mainthm}

See \S\ref{subsec: framed_mcg} for a uniform spinal open book and the generalized Hurwitz equivalence. There are two sub-cases of the above theorem where we can obtain exact or Stein fillings.

\begin{mainthm}\label{thm: panlf_construct_stein}
  Let $(M, \xi)$ be a contact 3-manifold supported by a planar spinal open book with page $P$ and monodromy $\phi$. Suppose the spinal open book is uniform with respect to $\mathbb{D}^2$, and that at least one of the branched cover maps $\pi_i$ has no branch points. Then any positive admissible factorization of $\phi$ with respect to $\mathbb{D}^2$ is Hurwitz equivalent to one induced by the nearly Lefschetz fibration structure of a Stein filling of $(M, \xi)$. 
\end{mainthm}

\begin{remark}\label{rem: steindeformation_equiv}
  We conjecture that the Stein structures on the fillings obtained in Theorem~\ref{thm: panlf_construct_stein} are unique up to Stein homotopy -- this would follow from the existence of a unique complex curve up to complex isotopy corresponding to a given quasipositive braid factorization of a quasipositive knot in $(S^3, \xi_{st})$. To our best knowledge, both \cite{rudolph2004algebraic} and \cite{hedden} give constructions of such complex curves, but it is unclear whether their constructions are complex isotopic. However the symplectic structure obtained in Theorem~\ref{thm: panlf_construct} is canonical since the symplectic surface corresponding to a quasipositive factorization of its boundary link is unique up to deformation.
\end{remark}

\begin{mainthm}\label{thm: panlf_construct_exact}
  Let $(M, \xi)$ be a contact 3-manifold supported by a planar spinal open book with page $P$ and monodromy $\phi$. Suppose the spinal open book is uniform with respect to $B$, and let $\rho: \pi_1(B) \to \Mod(P)$ be a monodromy representation. Suppose further that at least one of the branched cover maps $\pi_i$ has no branch points. Then any positive admissible factorization of $\phi$ with respect to $(B, \rho)$ is Hurwitz equivalent to one induced by the nearly Lefschetz fibration structure of an exact filling of $(M, \xi)$.
\end{mainthm}

\begin{proof}[Proof of Theorem \ref{thm: monodromy_fact}]
  This follows from Theorem \ref{thm: intro-localmodel}, Theorem \ref{thm: panlf_construct}, Theorem \ref{thm: panlf_construct_stein} and Theorem \ref{thm: panlf_construct_exact}.
\end{proof}

%---------------------------------------------------
\subsection{Symplectic fillings of torus bundles} 
%---------------------------------------------------
Ever since Eliashberg \cite{Eliashberg:T3} determined the strong fillability of contact structures on $T^3$, there has been extensive studies on the symplectic fillability of torus bundles. In particular, Stein and strong fillability of contact structures on torus bundles are completely determined except for negative hyperbolic ones, through extensive studies by \cite{DG:fillings,EH:nofillings,Gay:fillings,Van_horn_morris_thesis,GL:fillings,ding2018strong, Christian:fillings}. In addition, several classification results for these torus bundles are also known, see \cite{EH:nofillings,Christian:fillings, Kaloti:fillings, Lisca:lensFillings,spinal_II, PVHM:lensFillings} for examples.

In this section, we utilize our main results to complete the classification of strong symplectic fillings of negative parabolic torus bundles and some elliptic torus bundles. While these are already new results previously inaccessible, we expect the techniques developed in our paper to be of use for many other classification problems as well in the future.

We first review contact structures on torus bundles. Let $A \in SL_2(\Z) = \Mod(T^2)$. Then we define the torus bundle with monodromy $A$ by
\begin{align*}
  T_A := T^2 \times \R / (x,t) \sim (Ax,t-1) 
\end{align*}
$T_A$ is called {\bf elliptic, parabolic}, and {\bf hyperbolic}, respectively, as $|\mathrm{tr}(A)|$ is less than 2, equal to 2, or greater than 2. Tight contact structures on torus bundles were classified by Giroux \cite{Giroux:classification} and Honda \cite{honda2}. We consider a {\bf rotational} contact structure $\xi_n$ for an integer $n \geq 0$ defined by 
\begin{align*}
  \xi_n = \ker[\sin \phi_n(t)\,dx + \cos \phi_n(t)\,dy],\;\; (x,y) \in T^2, \,t \in \R / \Z
\end{align*}
where $\phi_n \colon \R/\Z \to \R$ satisfies
\[
  n\pi < \sup_{t \in \R / \Z}(\phi_n(t+1) - \phi_n(t)) \leq (n+1)\pi.
\]
We say $\xi_n$ has \textbf{$n \pi$-twisting}. Each torus bundle only admits either odd or even rotational contact structures. If $A$ is parabolic, then $A$ is conjugate to 
\[  
  \pm A_k = \pm\begin{pmatrix}1 & k \\ 0 & 1\end{pmatrix}
\]
for some $k \in \Z$. We say a parabolic torus bundle $T_A$ is {\bf positive} if $A$ is conjugate to $A_k$ and {\bf negative} if $A$ is conjugate to $-A_k$. Let $T_+(k) := T_{A_k}$ and $T_-(k) := T_{-A_k}$. Notice that $T_{+}(k)$ admits rotational contact structures $\xi_{2n}$ for $n \geq 0$ and $T_-(k)$ admits rotational contact structures $\xi_{2n+1}$ for $n \geq 0$. There are additional contact structures $\eta_n$ on a parabolic torus bundle $T_{\pm}(k)$ for $n \in \N$. 
\begin{align*}
  \eta_n = \ker \Big(|k+1|\cos(2\pi ny) dt + \sin(2\pi ny) dx - kt\sin(2\pi n y)dy \Big).
\end{align*}
In \cite{spinal_II}, strong fillings of contact structures on parabolic torus bundles were classified except for $(T_-(k), \xi_{1})$ as they do not admit a Lefschetz-amenable partially planar open book.  

We show that a rotational contact structure on a torus bundle with at least $\pi$-twisting admits a partially planar uniform spinal open book decomposition.  

\begin{theorem}\label{thm: torusbundle_spinal}
  Let $\xi_n$ be a rotational contact structure on a torus bundle $T_A$ with $n \geq 1$. Then it admits a uniform planar spinal open book with a connected page, disk vertebrae, and one annulus vertebra.
\end{theorem}

Theorem \ref{thm: torusbundle_spinal} equips these torus bundles with a planar spinal open book to which we can now apply Theorem \ref{thm: monodromy_fact}. Heuristically, this means that the classification question of the symplectic fillings of these torus bundles can be reduced to a monodromy factorization problem, modulo certain Hurwitz equivalences. However, as in Remark \ref{rem: equivalence of fillings}, we may only use the indirect equivalence relation going through the deformation classes of the symplectic fillings constructed from monodromies directly. In this circuitous path, we are able to classify fillings completely for the elliptic and parabolic torus bundles that fall in the above category. 

As mentioned above, $\pi$-twisting parabolic torus bundles $(T_-(n), \xi_1)$ were not accessible since they are supported by non-Lefschetz-amenable planar spinal open books \cite{spinal_II}. It was shown by Ding and Li \cite{ding2018strong}, extending work of Van Horn-Morris \cite{Van_horn_morris_thesis}, that they are Stein fillable if and only if $n \geq -4$. We provide an independent proof of this fillability result and give a full classification.

\begin{restatable}{theorem}{parabolic}\label{thm: parabolicbundles}
  The torus bundle $(T_-(n), \xi_1)$ is strongly fillable if and only if $n \geq -4$, and for $n \geq -4$, it admits a unique Stein filling up to symplectic deformation.
\end{restatable}

This theorem, combined with the results in \cite[\S1.6.3]{spinal_II}, completes the classification of strong fillings of parabolic torus bundles. 

An interesting corollary of the above result is that mapping classes of surfaces, which are monodromies of spinal open books with fixed vertebra types, that correspond to Stein fillable contact $3$-manifolds, do not form a monoid in $\SMod$, in contrast with the results of \cite{baldwin_monoid} for regular open books.

\begin{restatable}{corollary}{notMonoid}\label{cor: not_monoid}
  Let $P$ and $\Sigma$ be a page and a vertebra of a spinal open book, respectively. Assume further that $P$ and $\Sigma$ are annuli. Consider the subset $\Stein(P,\Sigma) \subset \SMod(P)$ consisting of monodromies of $P$ that makes the open book Stein fillable. Then $\Stein(P,\Sigma)$ is not closed under multiplication, hence not a monoid.
\end{restatable}

The elliptic torus bundles were shown to have unique fillings up to diffeomorphism by Golla and Lisca \cite{GL:fillings}. We improve their result to symplectic deformation equivalence.

\begin{restatable}{theorem}{ellipticbundles}\label{thm: ellipticbundles}
  Any rotational contact structure on an elliptic torus bundle with $\pi$-twisting admits a unique Stein filling up to symplectic deformation.
\end{restatable}

%-----------------------------------------
\subsection{Further examples}
%-----------------------------------------
We classify strong symplectic fillings of some contact 3-manifolds that are supported by non-Lefschetz-amenable planar spinal open books discussed in \cite{spinal_II}, specifically in Examples 1.33 and 1.34. The page monodromies of these examples are not supported in the interior of the page.

\begin{theorem}\label{thm:eg1_33}
  Consider the quotient contact manifold 
  \[
    (M, \xi) = (S^1 \times S^2, \xi_{st}) \big/ (t,\theta,\phi) \sim (-t, \pi - \theta, \phi + \pi)
  \]
  which is a non-orientable circle bundle over $\R P^2$ with an orientable total space. $(M,\xi)$ admits a unique Stein filling which is a PANLF, whose completion contains a single exotic fiber.
\end{theorem}

The next application concerns a contact 3-manifold obtained by taking the boundary of the quotient of a product symplectic manifold. The resulting contact $3$-manifold is supported by a planar spinal open book that is uniform with respect to two distinct surfaces. The construction is explained in \cite[Example 1.34]{spinal_II}, and we follow the same notations. Let $\psi_1$ and $\psi_2$ be involutions of $\Sigma_{2,2}$ arising as the deck transformations of the double branched covers $\phi_1: \Sigma_{2,2} \to \Sigma_{1,2}$ and $\phi_2: \Sigma_{2,2} \to \Sigma_{0,2}$, where $\phi_1$ is unbranched and $\phi_2$ has four simple branch points. Also let $\sigma$ be an involution $\sigma(s,t) = (-s,-t)$ on $\Sigma_{0,2} = [-1,1] \times S^1$. Then the following Weinstein domains 
\begin{align}\label{eq:E}
  E_i = (\Sigma_{2,2}\times \Sigma_{0,2}) \big/ (z,w) \sim (\psi_i(z), \sigma(w)).
\end{align}
induce the same boundary contact 3-manifold $(M,\xi)$, and it is supported by a spinal open book with two annulus pages. In addition, each annulus page $A_i$ for $i \in \{1,2\}$ is equipped with a monodromy exchanging its two boundary components, and these two boundary components are attached to the same boundary component $\partial_i$ for $i \in \{1,2\}$ of the spine with vertebra $\Sigma_{2,2}$.

It was shown in \cite{spinal_II} that $E_1$ admits a bordered Lefschetz fibration structure inducing the above planar spinal open book on the boundary, while $E_2$ does not. We claim that $(M,\xi)$ admits two types of fillings: one is a bordered Lefschetz fibration, and the other is $E_2$.

\begin{theorem}\label{thm:eg1_34}
  The manifold $(M, \xi)$ described above admits exactly two types of Stein fillings up to symplectic deformation: a family of Stein fillings that are bordered Lefschetz fibrations with $\Sigma_{1,2}$ base and annulus fibers, and the other is a PANLF with annulus base and annulus fibers, whose completion contains four exotic fibers. 
\end{theorem}

We further show that there exist symplectic Lefschetz fibrations with arbitrarily high genus fibers that also admit a PANLF with planar fibers. This family of examples also shows how certain open books with arbitrarily high genus can be converted into planar spinal open books, which makes understanding their fillings a more tractable problem. However these examples also highlights the difficulty of the classification problem in this setting, see Remark \ref{rem: high_genus_difficulty}.

\begin{theorem}\label{thm: highgenus}
  Consider a $3$-manifold supported by the genus-$g$ open book shown in Figure~\ref{fig:highgenus}, for $g \geq 2$. Each of these contact $3$-manifolds has a Stein filling with a Lefschetz fibration structure with genus $g$ regular fibers that induce a positive factorization of the genus-$g$ open book monodromy, but also a nearly Lefschetz fibration structure that induces a positive admissible factorization of the monodromy of a planar spinal open book supporting the same contact manifold.
\end{theorem}

%-----------------------------
\subsection{Organization} 
%-----------------------------
The background material on contact and symplectic topology, regarding spinal open books, bordered Lefschetz fibrations, PANLFs, multisections, and monodromy factorizations is covered in \S\ref{sec: background_top}. The background material on pseudoholomorphic curves and double completions is discussed in \S\ref{sec: pseudoholomorphic}. The proofs of Theorem~\ref{thm: intro-localmodel} and Theorem~\ref{thm: main} are given in \S\ref{sec:exotic_fibers}. In \S\ref{sec: constructions}, we prove Theorems~\ref{thm: panlf_construct}, \ref{thm: panlf_construct_exact}, and \ref{thm: panlf_construct_stein}. Then finally in \S\ref{sec: classification_filling} we explore the applications and prove Theorems~\ref{thm: parabolicbundles}, \ref{thm: ellipticbundles}, \ref{thm:eg1_33}, \ref{thm:eg1_34}, and Theorem~\ref{thm: highgenus}.

%------------------------------
\subsection{Acknowledgements} 
%------------------------------
We are grateful to Chris Wendl and Samuel Lisi for generous help in understanding their results and for crucial suggestions and input. We would also like to thank Mohammed Abouzaid, Inanc Baykur, Georgios Dimitroglou Rizell, Yakov Eliashberg, John Etnyre, Matt Hedden, Helmut Hofer, Michael Hutchings, Olga Plamenevskaya, Laura Starkston, and Jeremy Van-Horn Morris for helpful conversations over the course of this project. The project was significantly helped by conferences which allowed us to collaborate in-person -- we extend our thanks to the organizers of GSTGC 2022, Kylerec 2022, BIRS 2023 workshop on Interactions between Symplectic and Holomorphic Convexity in 4 Dimensions (23w5123), and the MSRI-PIMS 2022 Summer school on Floer Homotopy Theory. AR also acknowledges partial support from Georgia Tech and Louisiana State University, and the NSF grants DMS-1906414, DMS-2203312, and DMS-1907654. LW acknowledges support from NSF Grants DGE-2146752 and DMS-2303437 and IAS Giorgio and Elena Petronio Fellow II Fund.

%%%%%%%%%%%%%%%%%%%%%%%%%
\section{Background: Contact and symplectic topology}\label{sec: background_top}
%%%%%%%%%%%%%%%%%%%%%%%%%
In this section we review and establish some necessary background regarding low-dimensional contact and symplectic topology. We will assume the knowledge of open book decompositions, convex surfaces, Stein structures, and weak and strong symplectic fillings. The reader is invited to refer to the article \cite{ozbagci_survey} for background on symplectic fillings and Stein structures, and to the lecture notes \cite{etnyre_obd} for a thorough exposition on open book decompositions and convex surfaces.

%----------------------------------------------------------------------------
\subsection{Spinal open book decompositions and contact structures}\label{subsec: spinalobd}
%----------------------------------------------------------------------------
In this section we review the definition of spinal open book decompositions first appeared in \cite{spinal_I, spinal_II}. We focus on spinal open books supporting closed contact $3$-manifolds, the objects of which we aim to classify the symplectic fillings in this paper. 
\begin{definition}\label{def:spinalobd}
  A \textbf{spinal open book decomposition} of a closed oriented $3$-manifold $M$ is a decomposition $M = M_\Sigma \cup M_P$ (called the \textbf{spine} and \textbf{paper}, respectively) together with a pair of fibrations
  \begin{align*}
    \pi_\Sigma &: M_\Sigma \to \Sigma \\
    \pi_P &: M_P \to S^1
  \end{align*}
  where 
  \begin{enumerate}
    \item $\Sigma$ is a compact oriented surface whose connected components (called \textbf{vertebrae}) have nonempty boundary and $\pi_\Sigma$ is a trivial fibration with $S^1$ fiber;
    \item The fibers of $\pi_P$ are compact oriented surfaces whose connected components (called \textbf{pages} and denoted by $P$) have nonempty boundary and intersect transversely to $\partial M_P$. The intersection of $P$ and $M_\Sigma$ consists of fibers of $\pi_\Sigma$.
    \item For every component of $\partial M_P$, there exist local coordinates $(\phi, t, \theta) \in S^1 \times (-1,0] \times S^1$ on a collar neighborhood such that $\pi_P(\phi,t,\theta) = m\phi$ for some $m \in \mathbb{N}$. On every component of $\partial M_\Sigma$, there exist local coordinates $(s, \phi, \theta) \in (-1,0] \times S^1 \times S^1$ such that $\pi_{\Sigma}(s, \phi, \theta) = (s,\phi)$. The number $m$ is called the {\bf multiplicity} of $\pi_P$ at that boundary component of $M_P$. On the overlap between $M_P$ and $M_\Sigma$ the 2-torus coordinates $(\phi, \theta)$ agree, while the interval coordinates are related by $s = -t$.
    \item The paper $M_P$ can be identified with a mapping torus 
    \[
      M_P = \R \times P \big/ \sim, \quad ( \tau, p) \sim (\tau-1,\phi(p))
    \]
      and $\pi_P: M_P \to S^1$ is given by $[(\tau, p)] \to [\tau]$. Here, $\phi \in \SMod(P)$ is the monodromy in the \textbf{spinal mapping class group} of $P$. See \S\ref{subsec: framed_mcg} for the definition.
  \end{enumerate}
\end{definition}

Note that when all vertebrae are disks and the multiplicity of each component of $\bd M_P$ is $1$, the definition of spinal open books recovers the definition of ordinary open books. A contact form $\lambda$ on a spinal open book is called a \textbf{Giroux form} if $d\lambda$ is positive on the interior of each page and the Reeb vector field associated to $\lambda$ is positively tangent to every fiber of $\pi_\Sigma$. We say that a contact structure $\xi$ on $M$ is \textbf{supported by $\pi$} if $\xi$ admits a Giroux form for $\pi$. 

\begin{remark}
  We will frequently use the $3$-tuple $(P, \phi, \Sigma)$ to refer to spinal open books, denoting the page, monodromy, and vertebra. This $3$-tuple does not uniquely determine a spinal open book, but is rather just a short-hand for notation.
\end{remark}

\begin{definition}\label{def:symmetricsobd}
  A spinal open book $\pi = (\pi_{\Sigma}, \pi_P)$ of $M$ is \textbf{symmetric} if 
  \begin{enumerate}
    \item all pages are diffeomorphic,
    \item for each vertebra $\Sigma_i$ $(\Sigma = \Sigma_1 \sqcup \cdots \sqcup \Sigma_r)$, there exists a number $k_i$ such that every page has exactly $k_i$ boundary components in ${\pi_\Sigma}^{-1}(\partial \Sigma_i)$.
  \end{enumerate} 
\end{definition}

We also say that $\pi$ is \textbf{uniform} if it is symmetric and there exists a compact oriented surface $\Sigma_0$ such that 
\begin{enumerate}
  \item connected components of $\partial \Sigma_0$ are in bijection with the components of $M_P$
  \item for each vertebra $\Sigma_i$ $(\Sigma = \Sigma_1 \sqcup \cdots \sqcup \Sigma_r)$, there exists a $k_i$-fold branched cover $\pi_i: \Sigma_i \to B$ such that for each $\gamma \subset \partial \Sigma_i$, $\pi_i|_{\gamma}$ is an $m_\gamma$-fold cover where $m_\gamma$ denotes the multiplicity of $\pi_P$ at $\pi_\Sigma^{-1}(\gamma) \subset M_P$.
\end{enumerate}

In the above case we further say that $\pi$ is \textbf{uniform with respect to $B$}, and we will refer to $B$ as the base of the spinal open book. We say that $\pi$ is \textbf{Lefschetz-amenable} if it is uniform and all branched covers have no branch points. This is the situation extensively studied in \cite{spinal_II}. In this paper, we focus on the non-Lefschetz-amenable cases that require us to consider exotic fibers. See \S\ref{sec:exotic_fibers} for more details on exotic fibers and their relations to the branch points.

%-----------------------------------------------------------------------
\subsection{Monodromy factorizations} \label{subsec: factorization}
%-----------------------------------------------------------------------
Suppose we have a spinal open book $(\pi_{\Sigma},\pi_P)$ of a closed $3$-manifold $M$ with (possibly disconnected) page $P$, monodromy $\phi$, and vertebrae $\Sigma = \Sigma_1 \sqcup \dots \sqcup \Sigma_r$. For $i = 1, \dots, r$, let $k_i$ denote the number of boundary components of $P$ that meet $\Sigma_i$. Let $B$ be a compact oriented surface with connected boundary such that for every $i$, there is a degree $k_i$ branched covering map $\pi_i :\Sigma_i \to B$ with $n_i$ simple branch points, and $(\pi_{\Sigma},\pi_P)$ is uniform with respect to $B$. Consider the following subset of $\partial P$:
\[
  \partial_B P := \{ c \in \partial P \mid c \text{ meets a spine component } \Sigma_i \times S^1 \text{ such that } n_i \neq 0\}.
\]
Fix a representation $\rho\colon \pi_1(B) \to Mod(P)$. Assume further that any image of $\rho$ is the identity near $\bd_B P$. Given such a representation, we can construct a fiber bundle over $B$ with fiber $P$. By restricting the base of the bundle to $\bd B$, we obtain a $P$-bundle over $S^1$ and denote the monodromy of this bundle by $\phi_{\rho} \in Mod(P)$. 

\begin{definition}\label{def: pos_adm_fact}
  Let $P$, $\phi$, $B$, $\rho$, and $n_i$ be as above. A {\bf positive admissible factorization of $\phi$ with respect to $(B, \rho)$} is given by expressing $\phi \circ \phi_{\rho}^{-1}$ as a product of
  \begin{enumerate}
    \item $n_i$ boundary interchanges between pairs of boundary components in $\partial_B P$ along arcs $\gamma_i$, joining the boundary components for $i=1,\ldots,r$, and
    \item positive Dehn twists about essential curves in $P$.
  \end{enumerate}
\end{definition}

\begin{definition}
  The collection of positive admissible factorizations of $\phi$ with respect to $(B, \rho)$, as $\rho$ runs over all possible representations, will be called the set of {\bf positive admissible factorizations of $\phi$ with respect to $B$}.
\end{definition}

In \S\ref{subsec: framed_mcg}, we will see that the mapping classes in (1) and (2) are the elements of the spinal mapping class group of $P$.

\begin{remark}\label{rem: boundaries}
  When there exist multiple page components, we do not need to consider factorization of the monodromy of each page component. First, recall that all page components are diffeomorphic since we assume the spinal open book is uniform. Let $P_1, \ldots, P_k$ be the page components, diffeomorphic to $P$, and $B$ a base of the spinal open book. Again, since the open book is uniform with respect to $B$, the page components are in one-to-one correspondence with the boundary components of $B$. Denote the monodromies of the page components $\phi_1, \ldots, \phi_k$. Let $\rho$ be a monodromy representation defining a fibration over $B$ with fibers and $\phi_{\rho,i}$ the monodromy on the corresponding page component induced by $\rho$. Let $\phi'_i = \phi_i\circ\phi_{\rho,i}^{-1}$. 
  We need to consider the factorization of each $\phi_i'$, but we can obtain a new spinal open book supporting the same contact 3-manifold by moving the monodromies across the page components. In fact, the new open book has the same pages and vertebrae. The new monodromies are 
  \begin{align*}
    &\phi^{\text{new}}_1 = \phi_k' \circ \phi_{k-1}' \cdots \circ\phi_1' \circ \phi_{\rho,1},\\  
    &\phi^{\text{new}}_2 = \phi_{\rho,2},\\
    &\dots\\
    &\phi^{\text{new}}_k=\phi_{\rho,k}.
  \end{align*}
  Therefore, we only need to consider the factorization of $\phi_k' \circ \phi_{k-1}' \cdots \circ\phi_1'$.
\end{remark}

%-----------------------------------------------------
\subsection{Bordered Lefschetz fibrations}
%-----------------------------------------------------
A spinal open book decomposition naturally arises as the boundary of a bordered Lefschetz fibration. In this section we recall the definition of a bordered Lefschetz fibration by following the conventions of \cite{spinal_I}. Let $E$ be a compact oriented connected 4-manifold with corners. We decompose the boundary of $E$ into two parts along the corners: 
\[
  \partial E = \partial_h E \cup \partial_v E.
\]
We further assume that the intersection of $\partial_h E$ and $\partial_v E$ is a collection of $2$-tori. 

Unlike \cite{spinal_I}, since we will not be carrying out any analysis near the corners, we will treat $E$ as a smooth manifold, assuming we already rounded the corners. 
Let $B$ denote a compact oriented connected surface with connected boundary.

\begin{definition}
  A \textbf{bordered Lefschetz fibration} of $E$ over $B$ is a smooth map $\Pi: E \to B$ with finitely many interior critical points $E^{\crit} \subset E^{\circ}$ and critical values $B^{\crit} \subset B^{\circ}$ such that
  \begin{enumerate}
    \item  $\Pi^{-1}(\partial B) = \partial_v E$ and $\Pi|_{\partial_v E} :\partial_v E \to \partial B$ is a smooth fiber bundle. %This corresponds to the paper of the boundary spinal open book.
    \item $\Pi|_{\partial_h E} :\partial_h E \to B$ is a smooth fiber bundle. %This corresponds to the spine on the boundary spinal open book.
    \item $(\Pi|_{\partial_h E}, \Pi|_{\partial_v E})$ is a spinal open book decomposition of $\bd E$. 
    \item For each $p \in E^{\crit}$ and $\Pi(p) \in B^{\crit}$, there are local holomorphic coordinates on $E$ centered at $p$ and on $B$ centered at $\Pi(p)$ such that $\Pi(z_1, z_2) = z_1^2 + z_2^2$.
    \item For each $z\in B$, the fiber $E_z := \Pi^{-1}(z)$ is connected and has nonempty boundary in $\partial_h E$.
  \end{enumerate}
\end{definition}

The fibers $\Pi^{-1}(z)$ for $z \in B \setminus B^{\crit}$ are called \textbf{regular fibers}, and for $z \in B^{\crit}$ are called \textbf{singular fibers}. The regular fibers are all homeomorphic smooth compact oriented surfaces with boundary. The singular fibers are smoothly immersed connected surfaces with positive transverse self-intersections -- topologically they are obtained by taking a regular fiber and pinching a curve on it down to a point; the curve that gets pinched is called a \textbf{vanishing cycle}. A Lefschetz fibration $\Pi$ is \textbf{allowable} if all the vanishing cycles are homologically essential curves on a regular fiber.

It is well-known that allowable bordered Lefschetz fibrations admit a Stein structure. 

\begin{theorem}{\cite[Theorem 3.9]{baykurvhm}\cite[Theorem B]{spinal_I}}\label{thm: blf_stein}
  Let $E$ be a 4-dimensional allowable bordered Lefschetz fibration and $\pi$ is an induced spinal open book of $\partial E$. Then, $E$ admits a (well-defined up to Stein homotopy) Stein structure which is a Stein filling of the contact manifold $(\partial E, \pi)$. Moreover, given two bordered Lefschetz fibrations $E_1$ and $E_2$ that fill $(\partial E, \pi)$, their Stein structures can be chosen to induce the same contact structure on the boundary.
\end{theorem}

%-------------------------------
\subsection{Nearly Lefschetz fibrations}\label{subsec: panlf} 
%-------------------------------
For a symplectic filling of an honest planar open book decomposition with suitable conditions, a pseudoholomorphic foliation gives a Lefschetz fibration \cite{Wendl_strongly}. However, symplectic fillings of a general spinal open book admit a generalized version of a Lefschetz fibration, which we call a \textbf{nearly Lefschetz fibration}. In particular, it gives rise to a uniform spinal open book on its boundary, which is non-Lefschetz-amenable in general.

\begin{definition}\label{def: panlf}
  A \textbf{nearly Lefschetz fibration} of $E$ over $B$ is a smooth map $\Pi: E \to B$ with finitely many interior critical points $E^{\crit} \subset E^{\circ}$ and critical values $B^{\crit} \subset B^{\circ}$, and finitely many \textbf{exotic points} $B^{\exot} \subset B^{\circ}$ such that:
  \begin{enumerate}
    \item $\Pi^{-1}(\partial B) = \partial_v E$ and $\Pi|_{\partial_v E} :\partial_v E \to \partial B$ is a smooth fiber bundle.
    \item For each $p \in E^{\crit}$ and $\Pi(p) \in B^{\crit}$, there are local holomorphic coordinates on $E$ centered at $p$ and on $B$ centered at $\Pi(p)$ such that $\Pi(z_1, z_2) = z_1^2 + z_2^2$.
    \item All fibers $E_z := \Pi^{-1}(z)$ for $z\in B$ are connected and have nonempty boundary in $\partial_h E$.
    \item For each $c \in B^{\exot}$, there are local holomorphic coordinates $\psi \colon U_c \to \C$, where $U_c \subset B^{\circ}$, sending $c$ to $0$ and $\phi\colon V_c \to (\C - D) \times \C$, where $V_c \subset N(\partial_h (E)) \cap \Pi^{-1}(U_c)$ and $D$ is a small neighborhood of $0 \in \C$, such that the following diagram commutes
    \[
    \begin{tikzcd}
      V_c \arrow{r}{\phi}  \arrow{d}{\Pi} &(\C - D) \times \C\arrow{d}{\Pi_c}\\
      U_c \arrow{r}{\psi} & \C
    \end{tikzcd}
    \]
    where $\Pi_c(z_1,z_2) = z_2^2-z_1$.
    \item The restricted projection $\Pi|_{\partial_h E \setminus (\cup_{c \in B_{\exot}} V_c)} :\partial_h E \setminus (\cup_{c \in B_{\exot}} V_c) \to B \setminus \cup_{c \in B_{\exot}} \psi(V_c)$ is a smooth fiber bundle, that extends to a smooth fiber bundle $\Pi_h: \partial_h E \to B$. 
  \end{enumerate}
\end{definition}

In the case of a nearly Lefschetz fibration, we have \textbf{regular} and \textbf{singular fibers} as in the case of a bordered Lefschetz fibration. The difference lies in the exotic points. To illustrate the topology of the neighborhood of an exotic point, let $c\in B^{\exot}$. For $z \in U_c \subset B^o$, $\Pi^{-1}(z)$ has the topology of a regular fiber near the boundary of $U_c$, but has one fewer boundary component in a neighborhood of the origin in $U_c$. We call all the fibers that have one fewer boundary component the \textbf{exotic fibers}.\footnote{In \cite{spinal_II}, the authors discuss exotic fibers in the context of completed fillings, and a finite number of exotic fibers are present in the completed filling. However, since we are looking at compact fillings, cutting off a cylindrical end, we obtain a disk worth of exotic fibers for every single exotic fiber in the completed filling.} \S\ref{sec:exotic_fibers} will discuss the topology of these exotic fibers in more details and identify them with the exotic fibers arising from the pseudoholomorphic foliations defined in \cite{spinal_II}.

A nearly Lefschetz fibration $\Pi$ is \textbf{allowable} if all vanishing cycles are homologically essential curves on a regular fiber. In analogy with positive allowable Lefschetz fibrations that are studied in the literature and denoted PALFs, we will also refer to these structures as positive allowable nearly Lefschetz fibrations and abbreviate them to PANLFs. 

\begin{remark}\label{rem: exotic_nbd_branch}
  Under the biholomorphism $(z_1, z_2) \mapsto (z_2^2 - z_1, z_2)$ from $\C^2 \setminus \{(z^2,z)\}$ to $\C \setminus \{0\} \times \C$, we can identify the neighborhood of an exotic fiber in the above, with the complement of a branch point of a multisection in a Lefschetz fibration as described in Definition~\ref{def: multisection}.
\end{remark}

%--------------------------------------------------------------
\subsection{Positive multisections and spinal mapping classes} \label{subsec: framed_mcg}
%--------------------------------------------------------------
Consider a compact connected oriented surface $F_g^n$ with genus $g$ and $n$ boundary components. We define its \textbf{spinal mapping class group} $\SMod(F_g^n)$ to be a subgroup of $\Mod(F_g^n)$ generated by mapping classes in $\Mod_{\partial}(F)$ (which are identity near $\partial F$) and mapping classes that interchange two boundary components along an arc in a counterclockwise way and rotates each boundary in a clockwise way. See Figure~\ref{fig:boundarytwist} for an example. Denote by $\tau_{\gamma}$ the {\bf boundary interchange} mapping class along the arc $\gamma$.

\begin{remark}\label{rem:framed_mcg}
  The boundary interchange mapping class can be interpreted as an element of the framed mapping class group defined by Baykur and Hayano \cite{baykur2016multisections}. The difference arises because we trivialize the bundle at the page boundary components, whereas they track the bundle at the boundary components of the fibers using marked points. The reader should observe that Figure \ref{fig:monodromy1} in our paper agrees with Figure 1 in \cite{baykur2016multisections}. 
\end{remark}

\begin{definition} \label{def: multisection}
  Let $\Pi\colon E \to B$ be a bordered Lefschetz fibration. A properly embedded compact oriented surface $\Sigma \subset E$ is called a {\bf positive multisection} or a {\bf positive $n$–section} of $E$ if it satisfies the following conditions: 

  \begin{enumerate}
    \item The restriction $\Pi|_{\Sigma}$ is an $n$–fold simple branched covering map.
    \item For any branch point $p \in  \Sigma \setminus \crit(\Pi)$, there exist local coordinates $\phi: U \to \C^2$ of $p$ and local coordinates $\psi: V \to \C$ of $q = \Pi(p)$ such that the following diagram commutes: 
    \[
    \begin{tikzcd}
      (U, U \cap \Sigma) \arrow{r}{\phi}  \arrow{d}{\Pi} &(\C^2, \Sigma_0)\arrow{d}{f_0}\\
      V \arrow{r}{\psi} & \C
    \end{tikzcd}
    \]
    where $f_0: \C^2 \to \C$ is projection to the first component, and $\Sigma_0 = \{(z^2,z)\} \subset \C^2$.
    \item If a branch point $p \in \Sigma$ of $\Pi|_{\Sigma}$ is in $\crit(\Pi)$, then there are complex coordinates $(U,\phi)$ and $(V, \psi)$ as above such that $\phi(U \cap \Sigma)$ is equal to $\{(z,z) \in \C^2\}$. 
  \end{enumerate}
\end{definition}

The following theorem restates \cite[Theorem 1.1]{baykur2016multisections} in the context of bordered Lefschetz fibrations (refer to Remarks~\ref{rem: exotic_nbd_branch} and \ref{rem:framed_mcg} for translation from the original statement).  
We will restrict our attention to multisections that have no branch points in $\crit(\Pi)$, which can always be arranged by a symplectic deformation. 

\begin{theorem}{\cite[Theorem 1.1]{baykur2016multisections}} \label{thm: bh_multisections}
  Let $\Pi\colon E \to B$ be a positive allowable bordered Lefschetz fibration with $k$ critical points and regular fiber $F$, and $\rho\colon \pi_1(B)\to \Mod(F)$ the monodromy representation for $E$. Suppose that $E$ induces a monodromy factorization $t_{c_1}t_{c_2}\dots t_{c_k}\in Mod_{\partial}(F)$ on $\partial E$ with respect to $(B,\rho)$, where $t_{c_i}$ is a positive Dehn twist about a vanishing cycle $c_i$. Consider a positive $n$-section $\Sigma \subset E$ with $m$ branch points disjoint from $\crit(\Pi)$. Then, the complement of a neighborhood of $\Sigma$ in $E$ admits a positive allowable nearly Lefschetz fibration, with a regular fiber $F^n$, which is $F$ with $n$ disks removed, and induces the positive allowable monodromy factorization $t_{c_1}t_{c_2}\dots t_{c_k}\tau_{\gamma_1}\tau_{\gamma_2}\dots \tau_{\gamma_m} \in \SMod(F^n)$ for some properly embedded arcs $\gamma_1,\ldots,\gamma_m$ in $F^n$.

  Conversely, from a positive allowable monodromy factorization $t_{c_1}t_{c_2}\dots t_{c_k}\tau_{\gamma_1}\tau_{\gamma_2}\dots \tau_{\gamma_m} \in \SMod(F^n)$, we can construct a positive allowable bordered Lefschetz fibration with a positive $n$-section.
\end{theorem}

We note that it is further true that any PANLF is the complement of a positive multisection in some bordered Lefschetz fibration. 

The following definition is due to Baykur and Hayano \cite{Baykur_Hayano_Hurwitz}.

\begin{definition}[Generalized Hurwitz equivalence] \label{def: generalized_hurwitz_equivalence}
  Let $F$ be a compact oriented surface with boundary. Two positive allowable monodromy factorizations of $\phi \in \SMod(F)$ will be said to be \textbf{Hurwitz equivalent} if they are related by the following moves:
  \begin{enumerate}
    \item {\bf Elementary transformation}, which changes a factorization as follows:
    \[
      \eta_{k+l}\cdots\eta_{i+1}\eta_i\cdots\eta_1 \longleftrightarrow \eta_{k+l}\cdots(\eta_{i+1}\eta_i\eta_{i+1}^{-1})\eta_{i+1}\cdots\eta_1
    \]
    \item {\bf Global conjugation}, which changes each member of a factorization by the conjugation of some mapping class $\psi \in Mod_\partial(F)$:
    \[
      \eta_{k+l}\cdots\eta_1 \longleftrightarrow (\phi\eta_{k+l}\phi^{-1})\cdots(\psi\eta_1\psi^{-1})
    \]
    \item {\bf Framing conjugation}, which changes a factorization as follows:
    \[
      \eta_{k+l}\cdots\eta_{i+1}\eta_i\eta_{i-1}\cdots\eta_1 \longleftrightarrow \eta_{k+l}\cdots\eta_{i+1}(t_{\delta}\eta_it_{\delta}^{-1})\eta_{i-1}\cdots\eta_1
    \]
    where $\delta$ is a simple closed curve parallel to a boundary component of $F$.
  \end{enumerate}
\end{definition}

The generalized Hurwitz moves naturally arise as a consequence of choices made when obtaining a factorization from a PANLF. As a result, two equivalent monodromy factorizations give rise to two symplectic PANLF that are deformation equivalent.

% Will add topological properties of multisections

%-----------------------------------------------------
\subsection{The relative ordinary open book of an annulus vertebra}\label{subsec:relobd}
%-----------------------------------------------------
We recall a relative open book decomposition, introduced in \cite{Van_horn_morris_thesis}. Let $P$ be a compact oriented surface with boundary and $\phi \in \Mod_{\partial}(P)$. Consider the mapping torus 
\[
  P_{\phi} = P \times [0,1] / \sim
\]
where $(p,1) \sim (\phi(p),0)$ for $p \in P$. Now pick a partition of $\partial P$ into two sets $\partial_B$ and $\partial_T$. Each component of $\partial_B$ is called {\bf binding circles}, and each component of $\partial_T$ is called {\bf boundary circles}. We foliate the boundary tori $\partial_B \times S^1 \subset \partial P_{\phi}$ into $\{x\} \times S^1$, where $x \in \partial_B$. From this, we can construct a $3$-manifold $M$ with torus boundary components by collapsing the foliating circles in each component of $\partial_B \times S^1$. We say $(P,\phi,(\partial_B,\partial_T))$ is a {\bf relative open book decompostion} of $M$. If $\partial_T = \varnothing$, we obtain a closed manifold with a regular open book decomposition.

As in the case of regular open book decompositions, there is a contact structure compatible with a given relative open book decomposition.

\begin{definition}
  Let $M$ be a compact oriented $3$-manifold and $(P,\phi,(\partial_B,\partial_T))$ a relative open book of $M$. A contact structure $\xi$ on $M$ is said to be {\bf compatible with} $(P,\phi,(\partial_B,\partial_T))$ if there exists a contact form $\alpha$ where $d\alpha$ restricts to a symplectic form on each page, and the foliation given by $\xi$ on $\partial M$ agrees with the foliation given by boundary circles in $\partial_T$.
\end{definition} 

One may glue relative open books along boundaries to build a contact structure supported by the glued open book decomposition. 

\begin{theorem}{\cite[Proposition 3.0.7]{Van_horn_morris_thesis}}\label{thm: gluing_rel_obd}
  Let $\mathcal{O}' = (P',\phi',(\partial_B',\partial_T'))$ be a relative open book of a (possibly disconnected) manifold $M'$. Let $T_1$ and $T_2$ be two boundary tori of $M'$. Let $\psi \colon T_1 \to T_2$ be an orientation reversing homeomorphism which identifies the foliations given by boundary circles in $\partial_T'$ on both. Let $M = M' \big/ \sim$ be the manifold formed by identifying $T_1$ and $T_2$ via $\psi$. Also, let $\mathcal{O} = (P,\phi,(\partial_B, \partial_T))$ be the open book of $M$ obtained by gluing the two boundary circles on $P'$ corresponding to $T_1$ and $T_2$. Then there exists a contact structure on $M$ compatible with $\mathcal{O}$, unique up to an isotopy fixing the boundary, that restricts to the contact structure on $M'$ compatible with $\mathcal{O}'$.
\end{theorem}

Next, we identify a contact structure supported by a relative open book and a contact structure supported by a spinal open book restricted to a spine component. A spine component with an annulus vertebra is homeomorphic to $S^1 \times I \times S^1$, a thickened torus. Van-Horn-Morris \cite{Van_horn_morris_thesis} introduced relative open books of $S^1 \times I \times S^1$. They correspond to the set of letters $\{a,a^{-1},b,b^{-1}\}$ as depicted in Figure~\ref{fig:relativeBook}. These were used as building blocks for a relative open book decomposition supporting a contact structure on $T^2 \times I$, and an open book decomposition supporting a rotational contact structure on a torus bundle over $S^1$. See Figure~\ref{fig:vanhornmorris} for an example. A torus bundle obtained by gluing these pieces is denoted by the corresponding word in $a$, $a^{-1}$, and $b^{-1}$. According to \cite{hkm_rightveering}, any contact structure corresponding to a word containing $b$ would be overtwisted since the monodromy is not right-veering. We find a relative open book of a spine component with annulus vertebra that agrees on the boundary of the spine component with the page foliation of the spinal open book.

\begin{figure}[htbp]{
  \vspace{0.2cm}
  \begin{overpic}[tics=20]
  {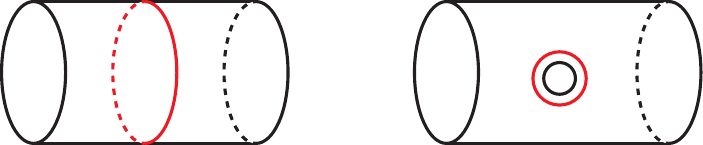}
  \end{overpic}}
  \vspace{0.2cm}
  \caption{Relative open book decompositions for $a$ and $b^{-1}$. For $a^{-1}$ and $b$ we use the negative Dehn twists.}
  \label{fig:relativeBook}
\end{figure}

\begin{figure}[htbp]{
  \vspace{0.2cm}
  \begin{overpic}[tics=20]
  {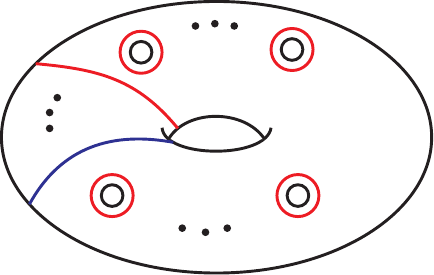}
  \end{overpic}}
  \vspace{0.2cm}
  \caption{An open book decomposition supporting a rotational contact torus bundle $(T_A,\xi_n)$.}.
  \label{fig:vanhornmorris}
\end{figure}

\begin{proposition}\label{prop:obd_ann_vertebra}
  The relative open book for a spine component of an annulus vertebra corresponds to the word $(aba)^{-2}$.
\end{proposition}

\begin{proof}
  A spine component with annulus vertebra is homeomorphic to $M = S^1 \times S^1 \times [0,1]$. Thus we can pick coordinates $(\phi, \theta, s)$ with $s \in [0,1]$, $\phi \in S^1$, and $\theta \in S^1$, following Definition~\ref{def:spinalobd}. We also orient $M$ using these coordinates. Now consider a relative open book $(P, \psi)$ of $M$. Then, the leaves of foliations on $\partial M$ given by the boundary circles of the relative open book should agree with the fibers of $\pi_\Sigma$. The foliations are given by 
  \[
    \bigcup_{\phi_0 \in S^1}\{(\phi_0, -\theta,0) \mid \theta \in S^1\}, \quad \bigcup_{\phi_0 \in S^1}\{(\phi_0, \theta,1) \mid \theta \in S^1\}.
  \]
  Thus the boundary circle foliation on $T^2 \times \{0\}$ rotates through an odd multiple of $\pi$ to obtain the foliation on $T^2 \times \{1\}$. By the construction of the Giroux form on $M_\Sigma$ in \cite[\S2.3]{spinal_I}, it follows that the boundary circle foliation rotates by $\pi$. 

  In \S4.4 of \cite{Van_horn_morris_thesis}, a correspondence is set up between words in $a$, $a^{-1}$, $b$, and $b^{-1}$ and the monodromy of the torus bundle, the element in $SL(2, \Z)$, under the identification 
  \[
    \Phi: \{(a,b) \mid aba = bab, (ab)^6 = 1\} \to SL(2,\Z).
  \]
  Van-Horn-Morris showed that $\Phi((aba)^{-4}) = \text{id}$ and by inserting the word $(aba)^{-4}$, we can add $2\pi$-twisting to a torus bundle without changing the monodromy. 
  Thus by inserting the word $(aba)^{-2}$, we can add $\pi$-twisting to a torus bundle, but change the monodromy by $-\text{id}$. Without identifying the two ends of $T^2 \times I$, the relative open book corresponding to the word $(aba)^{-2}$ gives the $\pi$-twisting contact $T^2 \times I$, which is contactomorhpic to an annulus spine component.
\end{proof}

%-------------------------------------
\subsection{Spine removal cobordisms} 
%-------------------------------------
\label{subsec:spine_removal}
Let $(M, \xi)$ be a contact manifold supported by a spinal open book decomposition $(S, \phi, \Sigma)$ and $\Sigma_{\text{rem}} \subset \Sigma$ an open and closed subset of the vertebra. Lisi, Van-Horn-Morris and Wendl \cite{spinal_I} define a weak symplectic cobordism called a \textbf{spine removal cobordism}, obtained by capping off the spine component $\Sigma_{\text{rem}}$.

\begin{theorem}{\cite[Theorem D]{spinal_I}}\label{thm: spineremoval}
  Consider the contact manifold $(\widetilde{M}, \widetilde{\xi})$, obtained from $(M, \xi)$ by deleting the spine component $\Sigma_{\text{rem}} \times S^1$ and capping off the corresponding boundary components of the pages by disks. Then there exists a symplectic cobordism with strongly concave boundary $(M, \xi)$ and weakly convex boundary $(\widetilde{M}, \widetilde{\xi})$, defined by attaching $\Sigma_{\text{rem}}\times \mathbb{D}^2$ with a product symplectic structure along $\Sigma_{\text{rem}} \times S^1$.
\end{theorem}

While a spine removal cobordism in general is only a weak symplectic cobordism, it gives a natural way to alter the spinal open books. We will use this cobordism in \S\ref{sec: constructions} when we construct  strong fillings of given spinal open books.

%%%%%%%%%%%%%%%%%%%%%%%%%%%%%%%%%%%%%%%%%%%%%%%%%%%%%%%%%%%%%%%%%%%%%%%%%%%%%%%%
\section{Background: Pseudoholomorphic curves}\label{sec: pseudoholomorphic}
%%%%%%%%%%%%%%%%%%%%%%%%%%%%%%%%%%%%%%%%%%%%%%%%%%%%%%%%%%%%%%%%%%%%%%%%%%%%%%%%
In this section we review some necessary background regarding pseudoholomorphic foliations and asymptotics of pseudoholomorphic curves. The reader is welcome to refer to \cite{ Wendl_low_dim_curves, Wendl_lecture_contact3, spinal_II} for a more detailed treatment of the material.

%------------------------------------------------------------------------
\subsection{Stable Hamiltonian structures and symplectic completions}  \label{subsec:shs}
%------------------------------------------------------------------------

\begin{definition}
  A \textbf{stable Hamiltonian structure} on a $(2n-1)$-manifold is a pair $\mathcal{H} = (\Omega, \Lambda)$ where $\Omega$ is a $2$-form and $\Lambda$ is a $1$-form satisfying $\Lambda \wedge \Omega^{n-1}>0$, $d\Omega = 0$ and $\ker \Omega \subset d\Lambda$.
\end{definition}

Similar to contact structures, we may consider the hyperplane field $\Xi:=\ker \Lambda$ and associate to $\mathcal{H}$ a unique vector field $R_\cH$, also called the \textbf{Reeb vector field}, given by $\Omega(R_\cH, \cdot)\equiv 0$ and $\Lambda(R_\cH) \equiv 1$. We may also define the \textbf{symplectization of $(M, \cH)$} to be $\bR \times M$ with the symplectic form
$$\omega_\phi:=d((e^{\phi(r)}-1)\Lambda) + \Omega$$
where $\phi$ is any smooth function $\phi: \bR \to (-\delta, \delta)$ with $\phi'>0$ and $\phi(r) = r$ for $r$ near $0$. The \textbf{symplectic completion} of a compact symplectic manifold $(W, \omega)$ with stable Hamiltonian boundary $\partial W = -M_-\sqcup M_+$ is the pair $(\widehat{W}, \omega_\phi)$ where
$$\widehat{W}:=((-\infty, 0] \times M_-)\cup W \cup ([0,\infty) \times M_+)$$
and
$\omega_\phi:=
\begin{cases}
  \omega_\phi:=d((e^{\phi(r)}-1)\Lambda_-) + \Omega_- & \text{ on } (-\infty, 0] \times M_-\\
  \omega & \text{ on } W \\
  \omega_\phi:=d((e^{\phi(r)}-1)\Lambda_+) + \Omega_+ & \text{ on } [0,\infty) \times M_+.
\end{cases}$

As in the contact case, we consider \textbf{compatible almost complex structure} $J$ on $\bR_r \times M$ with respect to the stable Hamiltonian structure $\cH:=(\Omega, \alpha)$. This means that $J \partial r = R_\cH$, $J(\Xi) = \Xi$ and $\Omega(\cdot, J\cdot)$ defines a bundle metric on $\Xi$. On the completion $(\widehat{W}, \omega_\phi)$, we consider the \textbf{compatible almost complex structure} $J$ to be such that $J|_W$ is compatible with $\omega$ and compatible with $\cH_-$ and $\cH_+$ on the symplectizations $(-\infty, 0] \times M_-$ and $[0,\infty) \times M_+$. Note that any such compatible $J$ is $\omega_\phi$-tame on $\widehat{W}$ for every $\phi\in \cT$. Therefore, the \textbf{energy of a $J$-holomorphic curve} $u:(S, j)\to (\bR\times M, J)$ defined as
$$E(u):=\sup_{\phi\in \cT} \int_S u^*\omega_\phi$$
is always greater or equal to zero, with equality exactly when the curve is constant. For more detailed introduction and references for stable Hamiltonian structures and curves in their symplectizations, see for example, \cite{Niederkruger_Wendl_weak_filling,Cieliebak_Volkov_stableHam, wendl2016lectures,Wendl_lecture_contact3, spinal_II}.

%------------------------------------------------------------------------------
\subsection{Moduli spaces of pseudoholomorphic foliations in a double completion} \label{subsec:moduli_space_foliation}
%------------------------------------------------------------------------------
Pseudoholomorphic foliations present an important tool in classifying symplectic fillings of contact manifolds via considering compactified moduli spaces of curves in their completions. In this section, we define the symplectic completion and moduli spaces considered in the proofs of Theorem \ref{thm: lvw_foliation} as well as our main theorems.

Given a strong filling of a contact manifold $(M,\xi)$ supported by a spinal open book, a compatible almost complex structure on a specific noncompact symplectic model is carefully constructed in \cite[\textsection 3]{spinal_II} in order to ensure the existence of pseudoholomorphic foliations. 

Specifically, given a spinal open book $M = M_p \cup M_\Sigma$, \cite[\textsection 3]{spinal_II} associates a  $(\widehat{E}, \omega_E)$ called a \textbf{double completion} and a weakly contact hypersurface $(M^0, \xi_0)$ contactomorphic to $(M, \xi)$. Here, an oriented hypersurface together with a co-oriented contact structure in a symplectic $4$-manifold is called \textbf{weakly contact} if the restricted symplectic form on the contact structure is positive. Topologically, the double completion takes the form
$$\widehat{E}:=(-1, \infty)_t \times \widehat{M}_\Sigma\cup (-1, \infty)_s\times \widehat{M}_P$$
where $\widehat{M}_\Sigma:=M_\Sigma\cup [0,\infty) \times \partial M_\Sigma$ and $\widehat{M}_P:=M_\Sigma\cup [0,\infty) \times \partial M_P$. The symplectic structure $\omega_E$ on $\widehat{E}$ is described in \textsection 3.3 in \cite{spinal_II}. In addition, a compatible stable Hamiltonian $\cH = (\Omega_0, \Lambda_0)$ can be endowed on $(M^0, \xi_0)$, which allows a nondegenerate perturbation to $(M^+, \cH_+ = (\Omega_+, \Lambda_+))$. Since we do not need the specific form of the symplectic structure or the stable Hamiltonian structure in our paper, we refer the interested reader to \cite[\textsection 3]{spinal_II} for details.

By construction, $\widehat{E}$ contains an unbounded closed subset $\widehat{N}_+(\partial E)\subset \widehat{E}$ which is identified with $[0,\infty) \times M^+$. Now, any $J_+$-compatible with $\cH_+$ determines an $\omega_E$-compatible almost complex structure on $[0,\infty) \times M^+ \cong \widehat{N}_+(\partial E)$. In \cite[\textsection 3]{spinal_II}, a particular $J_+$ is carefully chosen so that there exists a $\bR$-invariant foliation $\cF_+$ on $[0, \infty) \times M^+$ that is $J_+$-holomorphic. The construction of the foliation $\cF_+$ is in \textsection 3.8 in \cite{spinal_II}. Again, we do not need the specific form of the foliation $\cF_+$, but note that it includes ``holomorphic pages'' as shown in Proposition 3.12 in \cite{spinal_II}.

Now we discuss a specific model of the completion $(\widehat{W}, \widehat{\omega})$ of a symplectic filling $(W,\omega)$ which is a variation of the construction in \textsection \ref{subsec:shs}, but more suited for the discussion of pseudoholomorphic curves. Let $(M^-, \xi_-)$ be a weakly contact hypersurface obtained from $M^0$ by translation in the $s-$ and $t-$coordinates by a small negative value. In particular, $(M^-, \xi_-)$ is also contactomorphic to $(M,\xi)$. The completion is now given by
$$(\widehat{W}, \widehat{\omega}):= (W, \omega)\cup_{M^-} \big( \widehat{N}_-(\partial E), \omega_E\big)$$
where $\widehat{N}_-(\partial E)$ is the unbounded closed subset in $\widehat{E}$ with boundary $-M^-$ and the symplectic gluing is justified in \cite[\textsection 6.1]{spinal_II}. Note that $\widehat{N}_-(\partial E) \supset \widehat{N}_+(\partial E)$ and the almost complex structure $J_+$ constructed in the previous paragraph can be extended to an appropriate $\widehat{\omega}$-tame almost complex structure $\widehat{J}$ on $\widehat{W}$. Therefore, $(\widehat{W}, \widehat{J})$ contains a cylindrical end $([0,\infty) \times M^+, J_+)$ compatible with $\cH_+$.

Finally, we are ready to review the moduli spaces of interest defined in \textsection 6.1 of \cite{spinal_II}. Let $\cM(\widehat{J})$ be the moduli space of finite-energy holomorphic curves in $(\widehat{W}, \widehat{J})$ with arithmetic genus zero and $\overline{\cM}(\widehat{J})$ be its compactification. Furthermore, define
$$\widehat{\cM}(\widehat{J}):=\overline{\cM}(\widehat{J})/\sim,$$
where two curves are equivalent if and only if their asymptotic ends are the same up to permutation and their bottom-most nonempty levels are the same. It is easy to see that pseudoholomorphic leaves in the foliation $\cF_+$ are elements in $\widehat{\cM}(\widehat{J})$ as well. 

In fact, identify $[0,\infty) \times M^+$ and $\R \times M^+$ in the usual way, and let $\overline{\cM}^{\cF_+}_2(J_+)$ denote the compactified moduli space of unparameterized index $2$ $J_+$-holomorphic curves in $\bR \times M^+$ which belong to $\cF_+$ modulo $\bR$-translation, and
$$\widehat{\cM}^{\cF_+}(J_+):=\overline{\cM}^{\cF_+}_2(J_+)/\sim$$
with the same equivalence relation $\sim$ as above. Then $$\widehat{\cM}^{\cF_+}(J_+) \subset \widehat{\cM}(\widehat{J}).$$
Now define
$$\widehat{\cM}^{\cF}(\widehat{J}) \subset \widehat{\cM}(\widehat{J})$$
to be the smallest open and closed subset that contains $\widehat{\cM}^{\cF_+}(J_+)$. The main technical part of \cite{spinal_II} is to give a partition of the interior of $\widehat{\cM}^{\cF}(\widehat{J})$ in terms of regular, singular and exotic buildings. Define $\widehat{\cM}_{\reg}^\cF(\widehat{J})$, $\widehat{\cM}_{\sing}^\cF(\widehat{J})$ and $\widehat{\cM}_{\exot}^\cF(\widehat{J})$ to consist of building whose main levels are curves in $\cM_{\reg}(\widehat{J})$, $\cM_{\sing}(\widehat{J})$ and $\cM_{\exot}(\widehat{J})$, respectively, as in Theorem \ref{thm: lvw_foliation}.

\begin{theorem}{\cite[Proposition 6.3]{spinal_II}} \label{thm:spinal_ii_6.3}
  For a generic choice of $\widehat{J}$ on $\widehat{W}$ satisfying the above conditions, we have the following partition
  $$\widehat{\cM}^\cF(\widehat{J}) = \widehat{\cM}_{\reg}^\cF(\widehat{J})\cup \widehat{\cM}_{\sing}^\cF(\widehat{J}) \cup \widehat{\cM}_{\exot}^\cF(\widehat{J})\cup \widehat{\cM}^{\cF_+}(J_+),$$
  where $\widehat{\cM}_{\reg}^\cF(\widehat{J})$ is an open set and $\widehat{\cM}_{\sing}^\cF(\widehat{J})$ and $\widehat{\cM}_{\exot}^\cF(\widehat{J})$ are each a finite set. In addition, $\partial \widehat{\cM}^{\cF}(\widehat{J}) = \widehat{\cM}^{\cF_+}(J_+)$ and every point in $\widehat{W}$ is the image of the main level of a unique curve in the interior of $\widehat{\cM}^{\cF}(\widehat{J})$ via the continuous surjection
  $$\Pi\colon \widehat{W} \to \widehat{\cM}^\cF(\widehat{J})\setminus \widehat{\cM}^{\cF_+}(J_+): x \mapsto \text{the curve through } x.$$
\end{theorem}

The main technical contribution of our paper, detailed in \textsection \ref{sec:exotic_fibers}, is to give an explicit local model to the main level of an exotic building. This significantly improves the understanding of $\widehat{W}$ and its classification in terms of nearly Lefschetz fibrations.

%----------------------------------------------------------------------
\subsection{Asymptotic properties of pseudoholomorphic foliations}
%----------------------------------------------------------------------
In this subsection, we review the asymptotic properties of punctured pseudoholomorphic curves and behaviors of the moduli spaces described in \textsection \ref{subsec:moduli_space_foliation}. More details can be found in \cite{HWZI, HT_gluing1, Siefring_asymptotics}. Given a Reeb orbit $\gamma \subset Y$ and a pseudoholomorphic curve in $\R \times Y$ asymptotic to $\gamma$, one can naturally associate an \textbf{asymptotic neighborhood} $\R \times S^1 \times \R_\xi^2$ to this asymptotic end, where the choice of $\R_\xi^2$ depends on a trivialization of the contact structure $\xi$.

\begin{definition}[Asymptotic operator]
  The asymptotic operator\footnote{Our asymptotic operator is defined with an opposite sign from others in the literature, e.g. \cite{HWZI}.} $L_\gamma$ associated to a simply covered Reeb orbit $\gamma$ is a self-adjoint operator
  $$L_\gamma: = J_t \nabla^R_t: C^\infty(S^1, \gamma^* \xi) \to C^\infty(S^1, \gamma^* \xi),$$
  where $\nabla^R_t = \partial t + S_t$ is the symplectic connection determined by linearized Reeb flow, where $S_t$ is a symmetric matrix determined by $J(t)$ and $\omega$.
\end{definition}

Note that while asymptotic operators can be defined for multiply covered orbits, in this paper, we are only concerned with asymptotic operators associated to simple Reeb orbits. While exotic fibers contain an orbit that is doubly covered, in the local model in \S \ref{sec:exotic_fibers}, we focus on the behaviors of the underlying simple orbit in the nearby regular fibers.

Now fixing a unitary trivialization $\tau$, we may obtain a linearized Reeb flow $\Psi(t)$ and write 
$$L_\gamma = J_0 \frac{d}{dt} + S_t,$$
where $$S_t:= - J_0 \frac{d \Psi(t)}{dt} \Psi^{-1}(t)$$
is a symmetric matrix for each $t\in \R/\Z$.

Let $\eta(t)$ be an eigenfunction of $L_\gamma$ with eigenvalue $\lambda\in \sigma(L_\gamma)$. Then $\eta$ solves the ODE
$$\frac{d\eta(t)}{dt} = J_0(S_t - \lambda) \eta(t),$$
and hence $\eta$ is nonvanishing, if it's nonzero. Therefore, we may define $\wind_\tau(\eta)$ to be the winding number of the loop $\eta: \R/2\pi \Z \to \C$ around zero. As in \cite{Wendl_lecture_contact3}, we define the \textbf{extremal winding numbers} to be
\begin{align*}
  \alpha^{+}_{\tau}(\gamma)&:=\max\{\wind_{\tau}(\lambda)|\lambda\in \sigma(L_\gamma)\cap (0,\infty)\},\\
  \alpha^{-}_{\tau}(\gamma)&:=\min\{\wind_{\tau}(\lambda)|\lambda\in \sigma(L_\gamma)\cap (-\infty, 0)\}.
  %p(\gamma)&:= \alpha^{+}_{\tau} - \alpha^{-}_{\tau}.
\end{align*}

Now we review the \textbf{asymptotic expansion} of an asymptotically cylindrical pseudoholomorphic curve at a given nondegenenerate orbit. From now on, we focus on the case when a curve is positively asymptotic to a Reeb orbit, i.e. $s>>0$. The case for the negative asymptotic end is completely analogous.

Let $\gamma$ be a nondegenerate $T$-periodic Reeb orbit. Let $u$ be a $J$-holomorphic curve positively asymptotic to $\gamma$. Then by \cite{HWZII, Siefring_asymptotics}, for $(s,t)\in [R, \infty)\times S^1$ with some $R>>0$, we can write
\begin{equation}\label{eqn: hwz}
  \tilde{u}(s,t) = (Ts, \exp_{\gamma(t)}(\sum_{i=1}^N e^{-\lambda_i s} e_i(t) + o_{\infty,N})),
\end{equation}
where $\exp$ is the exponential map associated to a Riemannian metric on $Y$ and $o_\infty$ is a function $f: [R, \infty) \times S^1 \to \mathbb{R}^2$ satisfying the decay estimate   
\begin{equation}
\label{eq:error_estimate}
  |\nabla_s^i \nabla_t^j o_{\infty,N}(s,t)|\leq M_{ij} e^{-ds}
\end{equation}
for every $(i,j)\in \mathbb{N}^2$, where $M_{ij}$ and $d$ consist positive constants.

A very useful tool often providing ``rigidity'' phenomenon in symplectic geometry is pseudoholomorphic foliations. For this, one often assumes that the index of the pseudoholomorphic curves is $2$. If we further assume that the curves have an asymptotic end at an elliptic orbit, we can show the following local diffeomorphism. Recall that $c_N$ is the \textbf{normal Chern number} given by
$$c_N(u):=\frac{\ind(u) -2 + 2 g(u) +\# \Gamma_0(u)}{2},$$
where $\# \Gamma_0(u)$ denotes the number of asymptotic ends of $u$ with even Conley-Zehnder indices.

First, we recall ``automatic transversality'' in dimension four, which guarantees transversality for certain curves without genericity assumption on the almost complex structures. The general version is in \cite{Wendl_aut_tran}. However, we only need the following special case.

\begin{theorem}{\cite[Proposition A.1]{Wendl_strongly}}
\label{thm: auto_tran}
    If $u$ is an immersed finite energy pseudoholomorphic curve with $\ind(u)>c_N$, then $u$ is unobstructed.
\end{theorem}

\begin{lemma}[Local diffeomorphism]\label{lem:local_diff}
  Consider the moduli space $\mathcal{M}$ of index $2$ immersed curves with a simple asymptotic end at an elliptic  $\gamma$. Let $u\in \mathcal{M}$ be such that $c_N(u) = 0$. Then, there is a local diffeomorphism from $\text{ev}:\mathcal{M} \to \C\setminus \{0\}$ near $u$ at the asymptotic end $\gamma$ given by the evaluation map from the curve to its leading asymptotic eigenspace.
\end{lemma}

\begin{proof}
  This is an adaptation of the proof of Lemma A.2 in \cite{Wendl_strongly}. Since we are at an elliptic orbit, we may deform our almost complex structure via a path of admissible almost complex structure to $J'$ such that $S(t) = \theta$, after further deforming the trivialization. See for example \S2.6 of \cite{HT_gluing1}. Notice that both choices of deformations are contractible and that local diffeomorphisms are preserved under these deformations.

  Given $u\in \mathcal{M}$, recall that with respect to the splitting $u^* TW = T_u \oplus N_u$, we have a normal Cauchy--Riemann operator $D_u^N$ that is asymptotic to $L_\gamma$ and that $T_u \mathcal{M} = \ker D_u^N$. See \cite[\textsection 3.4]{Wendl_aut_tran}. By Theorem \ref{thm: auto_tran}, we also know that $\dim \ker D_u^N = 2$. For a curve asymptotic to $\gamma$, one can choose coordinates $(s,t)\in [0,\infty) \times S^1$ near the puncture asymptotic to $\gamma$ so that for large $s,t$, we have the linearized asymptotic formula by by \cite{HWZI} (see also \cite{Siefring_asymptotics}). Specifically, for any section $v\in \ker D_u^N$, we have that
  \begin{equation} \label{eq:local_expansion}
    v(s,t) = e^{\lambda s}(f_{v}(t) + r_v(s,t)),
  \end{equation}
  where $\lambda$ is an eigenvalue of $L_\gamma$, $f(t)$ is an eigenfunction of $L_\gamma$ with eigenvalue $\lambda$ and $r(s,t)\in \xi|_{\gamma}$ is smooth and converges uniformly to $0$ in $t$ as $s\to \infty$.

  Combining Equation (2.7) and Proposition 3.18 in \cite{Wendl_aut_tran}, since we are considering immersed curves, we have that
  \begin{equation} \label{eq:zz_cn}
    Z(v) + Z_\infty (v) = c_N(u),
  \end{equation}
  where $Z(v)$ is the algebraic count of zeros of $u$ and 
  \begin{equation}
    Z_\infty(v) := \sum_{\gamma\in \Gamma^+}[\alpha_{\tau}^-(L_\gamma)-\wind_{\tau}^\gamma(v)] +  \sum_{\gamma\in \Gamma^-}[\wind_{\tau}^\gamma(v) - \alpha_{\tau}^+(L_\gamma)].
  \end{equation}
    
  Since both $Z(v)$ and $Z_\infty(v)$ take nonnegative values, by the assumption $c_N(u) = 0$, we have that in particular $Z_\infty(v) = 0$, i.e. the extremal possible winding numbers are achieved, and in particular, $f_v$ belongs to the leading eigenspace at $\gamma$. 

  Furthermore, since we have deformed to the complex linear situation where $S(t) = \theta$, we know that the leading asymptotic eigenspace is $2$-dimensional. Therefore, $\text{ev}$ is a local diffeomorphism.
  %t \in [0,1]
\end{proof}

%%%%%%%%%%%%%%%%%%%%%%%%
\section{Exotic fibers} \label{sec:exotic_fibers}
%%%%%%%%%%%%%%%%%%%%%%%%

Let $M$ be a connected closed contact $3$-manifold with contact structure $\xi$ supported by a spinal open book $(\pi_\Sigma: M_\Sigma \to \Sigma, \pi_P: M_P \to S^1)$. In this section, we will first give a local model of a neighborhood of an exotic fiber and determine its local monodromy in terms of a boundary interchange \S\ref{subsec:local_model} and then show the uniqueness of such a local model up to symplectic deformation in an appropriate chart in \S\ref{subsec:unique_local_model}.

%--------------------------------------------
\subsection{A local model of exotic fibers} \label{subsec:local_model}
%--------------------------------------------
In this subsection we give a local model for the neighborhood of an exotic fiber arising as the SFT limit of a sequence of regular fibers. The setting is similar to that in Appendix of \cite{Wendl_low_dim_curves}, where an analogous model is given showing that Dehn twists characterize monodromies around a singular fiber in Lefschetz fibrations. Throughout this section, we will refer to the simple asymptotic orbit doubly-covered by an end of the exotic fiber as $\gamma_e$.

First, we identify the asymptotic neighborhood of an exotic fiber near the doubly-covered Reeb orbit with underlying simple orbit $\gamma_{e}$ with $(\mathbb{C}\backslash\{0\})\times \mathbb{C}$.

\begin{lemma}[Trivialization]\label{lem: orbit-nbd}
  The asymptotic neighborhood of a nondegenerate Reeb orbit $\gamma_{e} \subset (M,\xi)$, where $\xi$ is the stable Hamiltonian structure on $M$, admits a trivialization $\tau_0$ identified with $\R \times S^1 \times \R^2$, with a local almost complex structure and a stable Hamiltonian structure inherited from the identification with $(\C \backslash \{0\}) \times \C$ via the map
  \begin{equation} \label{eq:trivialization}
  \begin{split}
    \phi: (\C \backslash \{0\}) \times \C &\to \R \times S^1 \times \R^2\\
    \phi(r_1e^{i\theta_1}, r_2e^{i\theta_2}) &= (1/r_1, -\theta_1, r_1r_2e^{i(\theta_1 + \theta_2)})
  \end{split}
  \end{equation}
\end{lemma}

\begin{proof}
  From \cite{HWZ03} Example 3.1(4), we know that $(\C \backslash \{0\}) \times \C, i\oplus i)$ has cylindrical ends. The stable Hamiltonian structure associated to the level sets \[\{r_1e^{i\theta_1} \times \C | \theta_1 \in S^1\} \subset (\C \backslash \{0\}) \times \C\] is $\xi_{\theta_1} := \text{span}(\partial x, \partial y)$ where $x+iy = r_2e^{i\theta_2}$, and the 1-form is $d\theta_1$. Then, we may pick the trivialization identifying $\xi$ with $\phi_{*}\xi_{\theta_1}$.
\end{proof}

Now we discuss a \textbf{local model} of the foliation in a neighborhood of a doubly covered end of an exotic fiber. In Proposition \ref{prop:unique_local_model}, we will show that this local model is unique in an appropriate sense. 

\begin{definition}[Local model] \label{def:model} 
  Consider the sequence of curves
  $$u_c: \mathbb{C}\backslash\{\pm 1\} \to (\mathbb{C}\backslash\{0\}) \times \mathbb{C}$$
  defined by 
  \begin{equation}
    \label{eq:local_model}
    u_c(z)= (c(z^2-1), \sqrt{c}z).
  \end{equation}
  We see that as $c\to 0$, $u_c \mapsto (0,0)$, i.e., the limit does not exist, so we need to reparameterize. Let $z':= \sqrt{c}z$, so we have
  \begin{equation}\label{eq:u_c_z'}
    u_c(z') = ((z')^2-c, z').
  \end{equation}
  Now as $c\to 0$, $u_c \mapsto ((z')^2,z')$ and we see the double-covering of the orbit.
\end{definition}

\begin{remark}
  The moduli space of curves $u_c(z)$ in Equation~(\ref{eq:local_model}) is parameterized by the domain of $c\in \C\setminus\{0\}$. The reparameterization by $z':= \sqrt{c}z$ allows us to compactify the moduli space at $c = 0$ with a limiting building as in Figure \ref{fig:building}. The precise behavior of this SFT limit and the neighborhood of the fiber at $c=0$, which is in fact the exotic fiber, is detailed in the proof of Proposition \ref{prop:unique_local_model}. 
\end{remark}

\begin{definition}[Terminology after cutting off]\label{def:cpt_model}
  In light of Theorem \ref{thm: lvw_foliation}, we define the \textbf{compactified local model} to be the restriction of the above model to $(\C \setminus B) \times \C$, which corresponds to cutting off the cylindrical ends at the punctures. Here $B$ is defined to be the closed $\epsilon$-disk about the origin in $\C$. We look at the restricted foliation on $(\C \setminus B) \times \C$. For $u_c$, $\im(u_c) \cap (\C \setminus B) \times \C = \{u_c(z) | |z^2-1| > \frac{\epsilon}{|c|}\}$. Thus, for $|c| > \epsilon$,
  $\im(u_c) \cap (\C \setminus B) \times \C$ is topologically a pair of pants, while for $|c| < \epsilon$ $\im(u_c) \cap (\C \setminus B) \times \C$ is topologically an annulus, with the curves corresponding to $|c| = \epsilon$ having a singularity. So in the \textbf{compactified local model}, we get \textbf{regular curves}, corresponding to $\im(u_c)$ for $c > \epsilon$, and a $\mathbb{D}^2$-worth of \textbf{exotic curves}, corresponding to $\im(u_c)$ for $c \leq \epsilon$, which we call an \textbf{exotic neighborhood}.
\end{definition}

\begin{remark}
  The objective of the two definitions \ref{def:model} and \ref{def:cpt_model} is to have the vocabulary to talk about the neighborhood of the exotic fiber in the compact filling $W$, and separately in the completed filling $\widehat{W}$. 
\end{remark}

For now, we assume Proposition \ref{prop:unique_local_model}, i.e., that we have a unique local model around an exotic fiber and characterize the ``monodromy'' around an exotic fiber. To this end, we consider the projection map in a neighborhood of the exotic fiber in $\widehat{W}$ given in Proposition \ref{prop:unique_local_model},
\begin{equation}\label{eq:lefschetz_projection}
  \Pi_0: (\mathbb{C}\backslash\{0\})\times \mathbb{C} \to \mathbb{C} \text{ sending } (z_1,z_2)\mapsto z_2^2-z_1,
\end{equation}
whose level sets are identified with the pseudoholomorphic foliation near the positive asymptotic ends $\gamma_e \cup \gamma_e$.

We will see in Lemma~\ref{lem:monodromy} that this gives a model that interchanges the two punctures as one goes around $0\in \mathbb{C}$. Further, this local model of the exotic fiber also allows us to show that the number of exotic points is equal to the number of the branch points of $\Pi|_{\Sigma}$ in Theorem \ref{thm:branchpoints}. See also Remark \ref{rem:branchcovermaps}. To describe the monodromy of the fibers in a compactified filling around the exotic neighborhood, we will talk about \textbf{boundary components} of the fibers in the compactified model, which correspond to punctures in the local model.

\begin{lemma}\label{lem:monodromy}
  Consider the fibers in the compactified local model as in Definition~\ref{def:cpt_model}. Outside the exotic neighborhood, the fibers are pairs of pants. The monodromy of the bundle with pairs of pants fibers over $S^1$, around the exotic fiber $u_0$, is given by interchanging two boundary components in a counter-clockwise fashion along an arc and simultaneously rotating the boundary components by $\pi$ clockwise.
\end{lemma}

\begin{proof}
    
  In the model from Definition~\ref{def:model}, consider the following disk neighborhood $V$ of $0\in \bC$. Let $\delta>2\epsilon$ be large enough so that the $\delta$-disk neighborhood $ V\subset \bC$ contains the image of the exotic neighborhood. We parameterize $V$ so that $\Pi^{-1}(0)$ corresponds to the exotic fiber. Let $\zeta:= \partial V$. Over each point $x\in\zeta$, $\Pi^{-1}(x)$ is a pair of pants. We will characterize the monodromy of this bundle with pair of pants fibers around $\zeta$. Also consider a chart $U$ of the filling $W$ around the exotic neighborhood, as above. By construction, topologically $\Pi^{-1}(V) - U$ is a trivial $\Sigma - P$ bundle over the disk $V$, where $P$ is a pair of pants. In other words, the monodromy is supported in $P$.

  Now we parameterize $\zeta = \{\delta e^{2\pi i t}| t \in [0,1]\}$ and understand the monodromy around it. First, fix a trivialization on the fibers $\Pi^{-1}(\delta e^{2\pi i t})$, such that the bundle over $\zeta$ is trivial outside $P$. We can parameterize the fibers in $\Pi^{-1}(\zeta) \cap U$ as in Definition~\ref{def:model}, i.e. the trivialization is given by
  \begin{equation}
    z \mapsto (e^{2\pi i t}(z^2-1),e^{\pi i t}z)
  \end{equation}
  for $z \in \Pi^{-1}(\delta) \cap U$. Projecting to the second coordinate, the trivialization we need to consider is $H_t$ such that
  \[
    H_t: z \mapsto e^{\pi i t}z.
  \]
  We can parameterize $P$ as $\{|z|^2\leq2\} \cap \Pi^{-1}(\delta) \cap U$, so outside that we need to modify the above trivialization so that the bundle over $\zeta$ is trivial. Consider the function $f$ as shown in Figure~\ref{fig:trivialization}, and consider the trivialization
  \begin{equation}\label{eqn:monodromy_trivialization}
    h_t: z \mapsto e^{f(|z|)\pi i t}z.
  \end{equation}
    
  \begin{figure}[htbp]{
  \vspace{0.2cm}
  \begin{overpic}[tics=20]
  {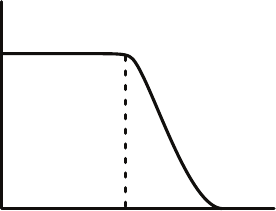}
   \put(-9, 74){\scriptsize$1$}
   \put(54, -10){\scriptsize$1.5$}
   \put(100, -10){\scriptsize$2$}
   \put(82, 100){\scriptsize$f(z)$}
  \end{overpic}}
  \vspace{0.2cm}
  \caption{The bump function needed to trivialize the monodromy for large $|z|$.}
  \label{fig:trivialization}
  \end{figure}

  \begin{figure}[htbp]{
  \vspace{0.2cm}
  \begin{overpic}[tics=20]
  {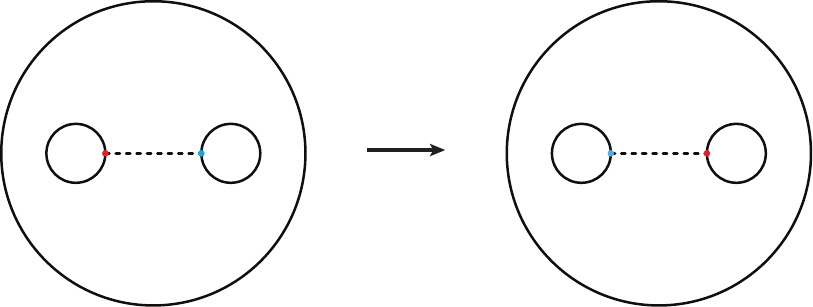}
   \put(35, 92){\scriptsize$1$}
   \put(109, 92){\scriptsize$2$}
   \put(72, 83){\scriptsize$\gamma$}
   
   \put(279, 92){\scriptsize$2$}
   \put(353, 92){\scriptsize$1$}
   \put(317, 83){\scriptsize$\gamma$}
  \end{overpic}}
  \vspace{0.2cm}
  \caption{Comparing the trivializations $h_0$ and $h_1$ after making the bundle trivial away from the pair of pants neighborhood.}
    \label{fig:monodromy1}
  \end{figure}
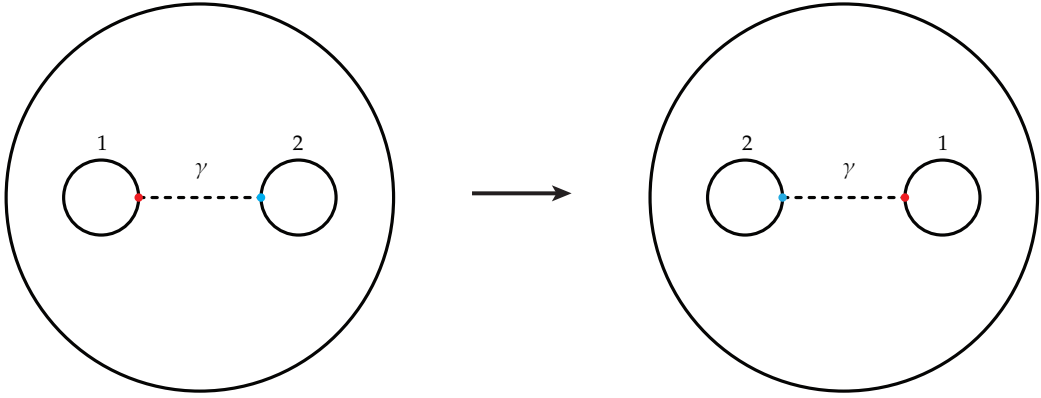

  Now, compare the trivializations $h_1$ and $h_0$. This is shown in Figure~\ref{fig:monodromy1} for the compactified model. 

  The above modification of the trivializations ensures that we have localized the monodromy into the pair of pants neighborhood on a fiber. But we are not done yet, since to understand the monodromy of the bundle, in the sense of (3) and (4) in Definition~\ref{def:spinalobd}, we further need to trivialize the bundle near the boundaries of the fibers, to agree with the model of $\pi_P$ near this component of $\partial M_P \subset \Pi^{-1}(\zeta)$. Observe that the component of $\partial M_P$ in this local model has multiplicity 2, again in the sense of Definition~\ref{def:spinalobd}. Let us refer to the component of $\partial M_P$ as $T$, and give it coordinates $(\phi, \theta)$, where $\theta$ is the $S^1$-coordinate for the boundary component on a page. Locally $\pi_P$ has the form $\pi_P|_T(\phi, \theta) = 2\phi$. However, the trivializations in Equation~\ref{eqn:monodromy_trivialization} do not agree on $T$, for different fibers and their intersections with $T$. To make them agree, a clockwise $2\pi$-rotation of each of the boundary components of $P$ needs to happen. Using a bump function as before, which is 0 on the boundary components and 1 just outside, the trivialization $h_1$ looks like as in Figure~\ref{fig:boundarytwist}. By an isotopy, we can see that the monodromy is exactly a positive half-twist and a half clockwise twist of each of the boundary components, as described in Figure~\ref{fig:boundarytwist}. 
\end{proof}

%---------------------------------------------
\subsection{Uniqueness of the local model} \label{subsec:unique_local_model}
%---------------------------------------------
In this subsection, we show that the local model given in Definition \ref{def:model} is unique up to deformation equivalence. This is obtained by a careful analysis of the asymptotic behaviors of pseudoholomorphic curves in a neighborhood of an exotic fiber, at the special elliptic orbit $\gamma_e$, which is doubly covered in the exotic fiber. The main technical argument is in Proposition \ref{prop:unique_local_model}. We start with the following lemma to prepare for computations in the asymptotic coordinates.

\begin{lemma}\label{lem:local_coord}
  The local model describes a foliation such that in the asymptotic coordinates near each puncture $\{\pm 1\} \in \bC$, the curves have leading terms being complex multiples of $e^{-it}$.
\end{lemma}

\begin{proof}
  To understand the asymptotic presentation of the curves $u_c$ near the punctures in the local model in Definition~\ref{def:model}, we will perform a change of coordinates on the domain from $\C \setminus \{\pm 1\}$ given by $z' := z-1$. This will give asymptotic coordinates near the puncture at 1, and it will be symmetric for the -1 case. We will also write the image of $u_c$ in $\R \times S^1 \times \C$ coordinates as in Lemma~\ref{lem: orbit-nbd}. The Reeb orbit in question, corresponding to $\{0\} \times \C$ in our previous coordinates or $\{\infty\} \times S^1 \times \{0\}$ in the current coordinates, has period 1 and in the model, the level $3$-manifold $\{s\} \times S^1 \times \C$  has the standard metric, hence we can use the trivial connection, hence the exponential map is just identity.

  Given $z \in \C \setminus \{\pm 1\}$, let $\phi_1 := \arg(c(z^2-1))$, $\phi_2 := \arg(\sqrt{c}z)$. Now in the asymptotic coordinates $(r_1, \theta_1, r_2, \theta_2) \in \R \times S^1 \times \C$, we have from \eqref{eq:local_model} that
  \begin{align*}
    u_c(z) &= (-\ln(|c(z^2-1)|), -\phi_1, |c\sqrt{c}(z^2-1)z|, \phi_1+\phi_2)\\
    &\quad \mbox{which gives after $z' = z-1$}\\
    u_c(z') &= (-\ln(|cz'(z'+2)|), -\phi_1, |c\sqrt{c}z'(z'+1)(z'+2)|, \phi_1+\phi_2).
    \end{align*}

  As we are using the trivial connection in $\{.\} \times S^1 \times \R^2$, the exponential map is identity under identification of $\xi_{\theta}$ with $\{.\}\times \{\theta\} \times \R^2$. So we can work with the above expression to understand the leading terms in the asymptotic expansion. The $\R \times S^1$ coordinate for the curve is $u_{c, \R \times S^1} = (-\ln(|cz'(z'+2)|), -\phi_1)$. Let us call that $(s,t)$. We then focus on the $\R^2$ coordinates, call it $u_{c, \R^2}$. We have that
  \[
    u_{c,\R^2}(r,\theta) = c\sqrt{c}z'(z'+1)(z'+2) e^{\phi_1+\phi_2}.
  \]
  Then $u_{c,\R^2}$ can be rewritten as  
  \begin{align*}
    u_{c,\R^2} &= e^{-s}\sqrt{c}e^{-it}(1+z')\\
    &= e^{-s}(\sqrt{c}e^{-it}+r(s,t))
  \end{align*}
  Thus, we see the leading term $\sqrt{c}e^{-it}$, corresponding to $e(t)$ in Equation~\ref{eqn: hwz}.
\end{proof}

Now we prove uniqueness of such a local model for an exotic fiber. To prepare for the proof for the uniqueness of the local model, we first establish some notations and Lemma \ref{lem: nearby_regular_fiber} detailing a few asymptotic properties of a regular fiber in the neighborhood of an exotic fiber.

Recall that $\widehat{\mathcal{M}}^{\cF}(\widehat{J})$ is the specific moduli space of pseudoholomorphic curves in $(\widehat{W}, \widehat{\omega}, \widehat{J})$ described in \textsection \ref{subsec:moduli_space_foliation}. In particular, it allows a partition into regular, singular and exotic moduli spaces together with $\widehat{\cM}^{\cF_+}(J_+)$ and there is a continuous surjection 
$$\Pi: \widehat{W} \to \widehat{\mathcal{M}}^{\cF}(\widehat{J}) \setminus \widehat{\cM}^{\cF_+}(J_+)$$ 
by Theorem \ref{thm:spinal_ii_6.3}. To establish notation, given a neighborhood $D \cong \C$ of an exotic fiber in the moduli space $\widehat{\cM}^{\cF_+}(J_+)$, we let $\widehat{u_c}:= \Pi^{-1}(c)$ where $c\in \C$ denote the pseudoholomorphic map from $\Sigma_c$ into $\widehat{W}$ corresponding to either a regular fiber when $c\neq 0$ and an exotic fiber when $c=0$.

For $c\neq 0$, consider now a twice-punctured disk $P \subset \Sigma_c$ for which there is an almost complex chart to $\C\setminus\{\pm 1\}$. In addition, identify the neighborhood $\widehat{N}\subset \widehat{W}$ of the elliptic orbit $\gamma_e$ with $(\C\setminus\{0\}) \times \C$ using Lemma \ref{lem: orbit-nbd}. The image of $\partial P$ under $\widehat{u_c}$, by continuity, naturally lands in the closure $\overline{\widehat{N}}$. Note that this closure can be understood as a product of $\C\setminus\{0\}$ and $\C$ both compactified at infinity, and the boundary of $\overline{\widehat{N}}$ belongs to the interior of $\widehat{W}$, away from the infinite end. In particular, the projection maps $\Pi_0$ and $\Pi$ on $\widehat{N}$ make sense in its closure.  

\begin{lemma} \label{lem: nearby_regular_fiber}
  For a regular fiber $\widehat{u_c}$ as above, $c_N(\widehat{u_c}) = 0$ and the extremal asymptotic winding number is achieved at each asymptotic end.
\end{lemma}

\begin{proof}
  The proof of Lemma 6.20 in \cite{spinal_II} implies that for regular curves $\widehat{u_c}$, we have
  $$\ind(\widehat{u_c}) + \#\Gamma_0(\widehat{u_c}) = 2.$$ 
  Therefore, $c_N(\widehat{u_c}) = 0$. Now, we may apply a similar trick as in the proof of Lemma \ref{lem:local_diff}: using Equation (\ref{eq:zz_cn}), we conclude that $Z_\infty(v) = 0$, i.e., the extremal winding numbers are achieved.
\end{proof}

\begin{proposition}[Uniqueness of local model] \label{prop:unique_local_model}
  Consider an exotic fiber in a completed filling $(\widehat{W}, \widehat{\omega},\widehat{J})$. Let $\widehat{N}$ be as above -- a noncompact neighborhood of the doubly covered orbit $\gamma_e$ in the exotic fiber. Then there exists a smooth coordinate chart $\phi:\widehat{N} \to \mathbb{C}\setminus \{0\} \times \C$ such that the intersection of the the nearby neighborhood regular fibers with $\widehat{N}$ is sent to the foliation $\mathcal{F}_0:=\{u_c|c\in \C \backslash\{0\}\}$, whose leaves are described in \eqref{eq:local_model} and are the non-zero level sets of the map 
  $$(z_1, z_2) \mapsto z_2^2 - z_1.$$
  It follows that the map $\Pi|_{\widehat{N}}$ is $\Pi_0$.
\end{proposition}

\begin{proof}
  Given $c\in \C\setminus\{0\}$, let $u_c := \widehat{u_c} |_P$ be a regular curve in a neighborhood of the exotic fiber in the moduli space $\widehat{\cM}^\cF(\widehat{J})$, restricted to a non-compact neighborhood of the two asymptotic end at $\gamma_e$ as above. By the continuous surjection in Proposition \ref{thm:spinal_ii_6.3}, or more directly, by the SFT compactness in the proof of Proposition \ref{thm:spinal_ii_6.3}, $u_c$ can be parameterized by maps of the form 
  $$u_c: \C\setminus\{\pm 1\} \to (\C\setminus\{0\}) \times \C,$$
  given by 
  $$u_c(z) = (q_c(z), p_c(z)),$$
  in a non-compact neighborhood of the collapsing punctures, where $f_c$ and $p_c$ are both holomorphic functions on $\C \setminus \{\pm 1\}$. See Figure \ref{fig:building}.

  Now, by assumption of the asymptotic neighborhood, $q_c$ can be extended to have removable singularities at $\pm 1$ with value $0$ and is non-vanishing on $\C \setminus \{\pm 1\}$. This means $q_c$ is a multiple of $(z^2 - 1)$, call it 
  $$q_c(z) = \alpha(c)(z^2 - 1),$$ 
  where $\alpha: \C \setminus \{0\} \to \C \setminus \{0\}$ is a smooth function in $c$. The $p_c$ coordinate is determined by the asymptotics of the holomorphic curves near the two ends at $\pm 1$, whereas the coordinate $q_c$ determines a simultaneous identification between $z \in P$ and $(s,t)$  as in Equation~(\ref{eqn: hwz}), near $\pm 1$, for all the curves $u_c$ for $c \in \C \setminus \{0\}$. 

  Note that the trivialization in the asymptotic neighborhood of $\gamma$ is given by $\phi$ in Lemma \ref{lem: orbit-nbd}, where $\gamma$ is the underlying simple elliptic orbit of the double cover $\gamma^2$ that appears in a positive asymptotic end of an exotic fiber. 

  We will now prove the proposition in three steps.

  {\bf Step 1 (change of coordinates):} First, since our orbit $\gamma$ is an elliptic orbit, we know that two leading asymptotic eigenfunctions have the same winding number. Therefore, we may change coordinates so that the leading asymptotic eigenfunctions becomes a complex multiple of $e^{i\Theta t}$. In fact, we can compute $CZ_\phi(\gamma) = 2\lfloor -1 \rfloor +1 = -1$ from our choice of trivialization in Equation (\ref{eq:trivialization}), so $\text{wind}_\phi(\gamma) =  -1$ because the leading eigenvalue is achieved by Lemma \ref{lem: nearby_regular_fiber}. In particular, we in fact have that $\Theta = -1$. 

  In addition, by Lemma \ref{lem: nearby_regular_fiber}, we also have that $c_N(\widehat{u_c}) = 0$. Therefore, we may apply Lemma \ref{lem:local_diff} to obtain a local diffeomorphism between the moduli space of regular curves in the neighborhood of an exotic fiber at the asymptotic end $\gamma$ and $\mathbb{C} \setminus \{0\}$. The local diffeomorphism allows us to pick asymptotic cylinders that we call $Z_c^1$ and $Z_c^{-1}$ in the neighborhood $D$ of an exotic fiber in the moduli space $\widehat{\mathcal{M}}^{\cF}(\widehat{J})$,  for which the leading eigenfunctions have coefficients $ \sqrt{c}$ and $-\sqrt{c}$, respectively.

  {\bf Step 2 (first deformation):} Observe that Lemma~\ref{lem:local_coord} also shows that our local model consists of curves with leading eigenfunction coefficients $\pm \sqrt{c}$, where the $\sqrt{c}$ coefficient corresponds to the asymptotic end of the curve $u_c$ near the puncture at +1, while $-\sqrt{c}$ corresponds to the the asymptotic end of $u_c$ near the puncture at $-1$.

  Now we modify the foliation $\widehat{\mathcal{M}}^{\cF}(\widehat{J})$ smoothly on a neighborhood $N \subset \widehat{N}$ near the asymptotic ends of the curves. This can be done by a similar analysis of Lemma 5.5 in \cite{Wendl_strongly}. Choose $R_1$ sufficiently large so that the tangent spaces of the curves in $N$ is uniformly close to the tangent space of the asymptotic orbit cylinders in $[R_1, \infty) \times Y$. Now we modify the rest of the asymptotic terms of $Z_c^1$ other than the leading term by cutoff functions to match the expansion in Lemma~\ref{lem:local_coord}, for the curve with leading eigenfunction coefficient $\sqrt{c}$. As in Lemma 5.5 in \cite{Wendl_strongly}, the modified foliation is still $J'$-holomorphic, for some $J'$ uniformly closed to $\widehat{J}$, so it's also tamed by $\widehat{\omega}$. Let $\omega_\epsilon$ be a symplectic form tamed by $J_\epsilon$. We then have that the modification in the pseudoholomorphic foliation is induced by a symplectic deformation $\omega_t = (1-t) \widehat{\omega}+t \omega_\epsilon$.

  Simultaneously, for $Z_c^{-1}$, we modify the rest of the asymptotic terms to match the expansion in Lemma~\ref{lem:local_coord}, for the curve with leading eigenfunction coefficient $-\sqrt{c}$. Doing this for all $c$ means that near $\pm 1$, the curves $u_c$ agree with the curves in Lemma~\ref{lem:local_coord}. By analytic continuation, it follows that they agree over the whole domain. 

  {\bf Step 3 (second deformation):}
  At this point, we know that after a reparameterization, all the curves in the neighborhood $\widehat{N}$ are parameterized as $u_c(z) = (\alpha(c)(z^2-1), \sqrt{c}z)$. This is a foliation, which means that given $(z_1, z_2) \in (\C \setminus \{0\}) \times \C$, there exists a unique pair $(z, c) \in (\C \setminus \{\pm 1\}) \times (\C \setminus \{0\})$ such that $u_c(z) = (z_1,z_2)$. This implies there is a unique solution in $c$ to the expression
  \[
    \alpha(c)z_2^2 - \alpha(c)c - z_1c = 0
  \] 
  for any $(z_1, z_2) \in (\C \setminus \{0\}) \times \C$.

  Now we argue that we can make $\alpha(c) = c$ via a symplectic deformation. First, we deform $\widehat{J}$, which induces a deformation of the foliation, in the $z_1$ direction via a $C^\infty$-small perturbation so that $\alpha(c)$ becomes a polynomial. By the arguments in the previous step, we have that this perturbation is also induced by a symplectic deformation. 

  Now, since $\alpha(c)$ is a polynomial and $\lim_{c \to 0}\alpha(c) = 0$ corresponding to the SFT limit, we may factor $\alpha(c) = c(\alpha_1(c))$ and obtain that 
  $$\alpha_1(c)z_2^2-\alpha_1(c)c-z_1 = 0$$ 
  has a unique solution in $c$ for all $(z_1,z_2)$. This means that $\alpha_1(c)$ must be a constant, and then by global scaling we may assume that $\alpha_1(c) = 1$. This concludes the proof.
\end{proof}

\begin{figure}[htbp]{
  \vspace{0.2cm}
  \begin{overpic}[tics=20]
  {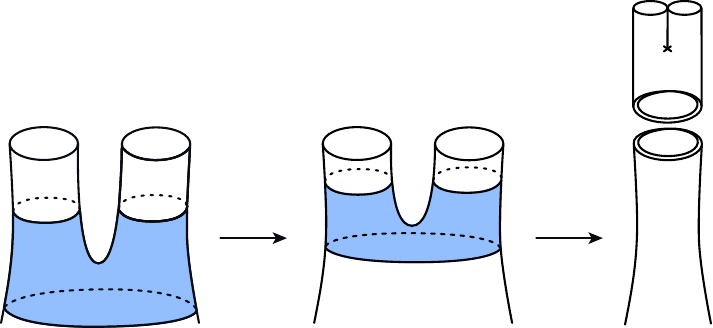}
  \end{overpic}}
  \vspace{0.2cm}
  \caption{A local description of the SFT limit from regular fibers to an exotic fiber.}
  \label{fig:building}
\end{figure}

\begin{lemma}\label{lem:intersect}
  In the chart $\widehat{N} \subset \widehat{W}$ as above, there is a unique non-transverse intersection of a fiber of the foliation with a representative of a vertebra, still denoted as $\Sigma$, near an exotic fiber, which corresponds to a simple branch point of the map $\Pi|_{\Sigma}: \Sigma \to \mathcal{M}$.
\end{lemma}

\begin{proof}
  Parameterize $\widehat{N}$ by $(r, \theta, p,q) \in \R \times S^1 \times \R^2 \cong (\mathbb{C}\backslash\{0\}) \times \C$ by the identification via Lemma \ref{lem: orbit-nbd}. In particular, as $r\to 0$, $t\to\infty$. Choose $\epsilon>0$ small enough so that $\widehat{N}_{r \leq \epsilon}$ corresponds to the symplectization in the double completion $\widehat{E}$  as described in \S\ref{subsec:moduli_space_foliation}, where all the curves in the foliation are cylinders over their asymptotic ends. Since $\widehat{N}$ is the symplectization over a neighborhood of a Reeb orbit, by the discussion in \cite[\S3]{spinal_II}, $\widehat{N}_{r \leq \epsilon} \subset (-1,\infty) \times M_{\Sigma} \subset \widehat{E}$. Comparing the symplectic structures and stable Hamiltonian structures in the local model $\widehat{N}$ and in $(-1,\infty) \times M_{\Sigma} \subset \widehat{E}$, it follows that the sets $\{r_0, \theta_0\} \times \C$ can be identified with the vertebrae in $\{r_0\} \times M_\Sigma$. Therefore, the intersection of the vertebrae with $\widehat{N}$ are holomorphic, and the projection map $\Pi: \widehat{W} \to \mathcal{M}$ takes the form $(z_1, z_2) \mapsto z_2^2 - z_1$.
 
  We now fix a representative $\Sigma$ of the vertebra, whose intersection with $\widehat{N}$ is $\Sigma_0 := \{r = r_0, \theta = 0\}$ where $r_0<\epsilon$. We now examine the intersections of $\Sigma_0$ with the fibers $\im(u_c)$. Parameterizing $\mathbb{C}\setminus \{\pm 1\}$ by $(x,y)$, with the natural identification $z = x+iy$, we see that 
  \[
    u_c = \begin{pmatrix} \Rea(\sqrt{c})x - \Ima(\sqrt{c})y\\ \Rea(\sqrt{c})y + \Ima(\sqrt{c})x\\ \Rea(c)(x^2-y^2-1) - \Ima(c)2xy\\ \Rea(c)2xy + \Ima(c)(x^2-y^2-1) \end{pmatrix}
  \]
  which implies that
  \[
    Du_c|_{(x,y)} =  \begin{pmatrix} \Rea(\sqrt{c}) & -\Ima(\sqrt{c})\\ \Ima(\sqrt{c}) & \Rea(\sqrt{c})\\ 2x\Rea(c)-2y\Ima(c) & -2y\Rea(c) - 2x\Ima(c)\\ 2y\Rea(c)+2x\Ima(c) & 2x\Rea(c)-2y\Ima(c) \end{pmatrix}.
  \]

  It follows that for any $c$, $\im(u_c)$ is transverse to $\Sigma_0$ at all points other than $u_c(0) = (-c,0)$. Therefore, the only non-transverse intersection of $\Sigma_0$ with a fiber $u_c$ happens when $u_c(0) \in \Sigma_0$, i.e., only when $c = -r_0$. 

  Since the projection map $\Pi$ in the completed filling is given by $(z_1, z_2) \mapsto z_2^2 - z_1$, we know that the restriction map $\Pi|_{\Sigma_0}$ locally has the form $z_2 \mapsto z_2^2 - r_0$, since we restrict at $\theta = 0$. This clearly has a unique branch point at $\{z_2=0\}$, corresponding to the non-transverse intersection point $(r_0,0)$ in the above paragraph.
\end{proof}

We now establish that there exists a representative $\Sigma_t$ of the vertebra $\Sigma$, such that the restriction $\Pi|_{\Sigma_t}$ is a simple branched cover, and there is a one-to-one correspondence between exotic fibers and branch points of $\Pi|_{\Sigma_t}$. 

\begin{theorem}\label{thm:branchpoints}
  Consider a contact $3$-manifold $(M, \xi)$ supported by a planar spinal open book $(P, \phi, \Sigma)$. Let $(W, \omega)$ be a symplectic filling of $(M, \xi)$ such that the completion $\widehat{W}$ admits a pseudoholomorphic foliation by Theorem~\ref{thm: lvw_foliation}. Then, there exists a representative $\Sigma_t$ of the vertebra such that $\Pi|_{\Sigma_t}$ is a branched cover, and the number of exotic fibers in the foliation is equal to the number of branch points.
\end{theorem}

\begin{proof}    
  In the double completion $\widehat{E}$ as in \S\ref{subsec:moduli_space_foliation}, consider $\Sigma_t := \{t,\theta\}\times \Sigma \subset \{t\} \times M_\Sigma$ for $t>>0$, i.e., a representative of the vertebra that lives high enough in the completed filling where all the curves in the foliation are asymptotic cylinders over their asymptotic orbits. Further assume that $t$ is high enough and we have chosen local coordinates so that for every exotic fiber, there is a neighborhood $\widehat{N}$ such that $\{t\} \times M_\Sigma \cap \widehat{N} \subset \widehat{N} \cap \widehat{E}$, and that $\Sigma_t \cap \widehat{N} = \Sigma_0$ as in the proof of Lemma~\ref{lem:intersect}. 
    
  In the proof of Proposition~\ref{prop:unique_local_model}, we considered a subset $D$ of the moduli space which was a neighborhood of the image of an exotic fiber. For every $c \in \mathcal{M}_{\exot}$, denote such a disk neighborhood of $c$ as $D_c$. In the proof of Proposition~\ref{prop:unique_local_model}, for every exotic fiber, we also considered a subset $\widehat{N} \subset \widehat{W}$ of the asymptotic Reeb orbit that the exotic fiber doubly covers. For every $c \in \mathcal{M}_{\exot}$, denote such an asymptotic neighborhood as $\widehat{N}_c$. We can assume that the image of the map $\Pi$ restricted to $\widehat{N}_c$ is $D_c$. Further, define $D_{\exot}$ and $W_{\exot}$ as follows:
  \begin{align*}
    D_{\exot} &:= \bigcup_{c \in \mathcal{M}_{\exot}} D_c,\\
    W_{\exot} &:= \bigcup_{c \in \mathcal{M}_{\exot}} \widehat{N}_c.
  \end{align*}
    
  By our choice of $t$, each fiber in the foliation intersects $\Sigma_t$ in their cylindrical ends approaching the asymptotic orbits. For any fixed curve in the foliation, its cylindrical ends can be modified slightly smoothly to agree with the trivial cylinder over the Reeb orbit, thus the intersections of $\Sigma_t$ with each cylindrical end is in bijective correspondence with the intersection of the vertebra with the Reeb orbit in the $3$-manifold.

  Given the above, we examine the map $\Pi|_{\Sigma_t}$ on two complementary subsets of $\widehat{W}$:

  \begin{enumerate}[I)]
    \item $\widehat{W} \setminus \Pi|_{\Sigma_t}^{-1}(D_{\exot})$ and
    \item  $\Pi|_{\Sigma_t}^{-1}(D_{\exot})$
  \end{enumerate}

  \noindent
  {\bf Case (i):} On $\widehat{W} \setminus \Pi|_{\Sigma_t}^{-1}(D_{\exot})$, there are no exotic fibers, thus all the cylindrical ends of any curve in the foliation restricted to this subset cover their asymptotic Reeb orbit once. This means that each curve intersects $\Sigma_t$ transversely, so it follows that $\Pi|_{\Sigma_t}$ restricted to 
  \[
    \Sigma_t \bigcap \bigg(\widehat{W} \setminus \Pi|_{\Sigma_t}^{-1}(D_{\exot})\bigg)
  \]
  is an honest covering map over $\mathcal{M} \setminus \mathcal{M}_{\exot}$. Suppose this is a $k$-fold covering map. 
    
  \noindent
  {\bf Case (ii):} Now, on $\Pi|_{\Sigma_t}^{-1}(D_{\exot})$, consider the subset $W_{\exot}$. Every cylindrical end of the foliation that does not intersect $W_{\exot}$ covers their asymptotic Reeb orbit once, by the same argument as the above paragraph. Thus $\Pi|_{\Sigma_t}$ is again an honest covering map restricted to $\Sigma_t \bigcap \bigg(\Pi|_{\Sigma_t}^{-1}(D_{\exot}) \setminus W_{\exot} \bigg)$. Since every curve in the foliation that lives in $\bigg(\Pi|_{\Sigma_t}^{-1}(D_{\exot}) \bigg)$ has exactly $(k-2)$ cylindrical ends in $\bigg(\Pi|_{\Sigma_t}^{-1}(D_{\exot}) \setminus W_{\exot} \bigg)$, this is a $(k-2)$-degree covering map. 
    
  By Lemma~\ref{lem:intersect}, in each neighborhood $\widehat{N}_c$, $\Pi|_{\Sigma_t}$ is a branched cover of degree 2, with exactly one branch point inside the neighborhood. Thus on $W_{\exot}$, the restriction of $\Pi_{\Sigma_t}$ is a degree 2 branched cover map with as many branch points as points in $\mathcal{M}_{\exot}$.

  Combining the above, we get that over $\Pi|_{\Sigma_t}^{-1}(D_{\exot})$, the map is a degree $k$ branch point with exactly $|\mathcal{M}_{\exot}|$ many branch points. 

  Since $\Pi|_{\Sigma_t}$ is smooth, it follows that it is a branched cover map over $\mathcal{M}$ with as many branch points as $|\mathcal{M}_{\exot}|$. 
\end{proof}

\begin{remark}\label{rem:branchcovermaps}
  Given a planar spinal open book uniform with respect to $B$, with spine $\Sigma = \cup_{i=1}^r \Sigma_i$, we thus have three branched cover maps from $\Sigma$ to $B$, given by 
  \begin{itemize}
    \item $\Pi|_{\Sigma_t}$ (defined above),
    \item $\displaystyle \bigcup_{\Sigma_i} \Pi|_{S_i}$ (defined in Theorem~\ref{thm: lvw_foliation}) , and 
    \item $\displaystyle\bigcup_{\Sigma_i}\pi_i$ (defined in Definition~\ref{def:symmetricsobd}).
  \end{itemize}  
  By the Riemann-Hurwitz formula, these maps have the same number of branch points. For this reason, we sometimes use the notation $\Pi|_{\Sigma}$ to implicitly denote $\Pi|_{\Sigma_t}$ for a representative $\Sigma_t$ constructed in Theorem \ref{thm:branchpoints}.
\end{remark}

Combining Lemma~\ref{lem:monodromy}, Proposition~\ref{prop:unique_local_model}, Theorem~\ref{thm:branchpoints}, and Remark~\ref{rem:branchcovermaps}, we have proved Theorem~\ref{thm: intro-localmodel}.

\begin{proof}[Proof of Theorem~\ref{thm: main}]
  It follows from Theorem \ref{thm: intro-localmodel} that any minimal strong filling $W$ of a planar spinal open book, where the compactified moduli space of curves foliating it is homeomorphic to $B$,  has the structure of a PANLF over $B$ inducing a positive admissible factorization of the page monodromy. Now, consider the union of the vertebrae for which $\pi_{i}$ has branch points, denoted $\Sigma_b$. Choose a filling $W$ as in the proof of Theorem~\ref{thm:branchpoints}, so that $\Pi|_{\Sigma_t}$ which is the restriction of the map $\Pi$ defining the PANLF, is also a branched cover map restricted to the vertebra $\Sigma_t$ in the boundary $\partial W$.
 
  We attach to $W$ a spine removal cobordism $C_{\Sigma_b}$ that caps off $\Sigma_b \subset \Sigma_t$. Extend the map $\Pi$ for the PANLF over $W \cup C_{\Sigma_b}$ by defining it to be $\Pi(x,y) = x$ for $(x,y) \in \Sigma_b \times D^2$; here we abuse notation by denoting the extended map $W \cup C_{\Sigma_b} \to B$ as $\Pi$. By Remark~\ref{rem: exotic_nbd_branch}, every exotic point local model $V_c$ as in Definition~\ref{def: panlf} can be identified with the complement of a thickened neighborhood of $S_0 = \{z^2,z\}$ in $\bD^4 \subset \C^2$. Locally around $V_c$, the cobordism $C_{\Sigma_b}$ amounts to adding $S_0$ back to obtain a $4$-ball neighborhood in the interior of $W \cup C_{\Sigma_b}$, denoted as $B_c$. Thus $\Pi$ as defined over $W \cup C_{\Sigma_b}$ only has Lefschetz type singularities, hence $W \cup C_{\Sigma_b}$ is a Lefschetz fibration over $B$. We now claim that $\Sigma_b \times \{0\} \subset C_{\Sigma_b}$ is a positive multisection of $\Pi: W \cup C_{\Sigma_b} \to B$.   
 
  Note that by the choice of $W$, the map $\Pi_{\Sigma_b \times \{0\}}$ is exactly the map $\Pi|_{\Sigma_b}$, hence it is a branched cover over $B$. Further, the only branch points occur in the neighborhoods $B_c$ where $S_0$ is exactly $\Sigma_b \times \{0\} \cap B_c$ and $\Pi$ agrees with the form in Definition~\ref{def: multisection}. Thus we are done.   
\end{proof}

%%%%%%%%%%%%%%%%%%%%%%%%%%
\section{Construction of strong fillings}\label{sec: constructions}
%%%%%%%%%%%%%%%%%%%%%%%%%%

The goal of this section is to construct a strong filling of a contact manifold corresponding to a positive admissible factorization of a given spinal open book monodromy. 

Recall that for a spinal open book $(P,\phi,\{\Sigma_i\})$, we denote by $k_i$ the number of boundary components of $P$ that intersect $\Sigma_i$ and by $B$ a compact oriented surface with connected boundary such that for every $i$, there is a degree $k_i$ branched covering map $\pi_i\colon \Sigma_i \to B$ with $n_i$ simple branch points. Also, let $\rho\colon \pi_1(B)\to \Mod(P)$ be a monodromy representation of the spinal open book. We further assume that the spinal open book is uniform with respect to $B$. The following theorem is a detailed version of Theorem~\ref{thm: panlf_construct}.

\begin{theorem}[Theorem \ref{thm: panlf_construct}]\label{thm: factorization_filling}
  Let $(M,\xi)$ be a closed contact $3$-manifold supported by a spinal open book decomposition $(P, \phi, \{\Sigma_i\}_{i=1}^m)$. Let $k_i$, $\pi_i$, $n_i$, $B$ and $\rho$ be as above. Then, for every Hurwitz equivalence class of positive admissible factorizations of $\phi$ with respect to $(B, \rho)$, there exists a symplectic PANLF $(W, \omega)$ with regular fiber $P$ and base $B$, such that $(W, \omega)$ is a strong filling of $(M,\xi)$ inducing the given spinal open book decomposition.
\end{theorem}

The proof of Theorem~\ref{thm: panlf_construct} shows that we obtain a slightly stronger result when at least one $\pi_i: \Sigma_i \to B$ has no branch points. 

The following theorem is a detailed version of Theorem \ref{thm: panlf_construct_exact}.

\begin{theorem}[Theorem \ref{thm: panlf_construct_exact}]\label{thm: exact_filling}
  Suppose $(M,\xi)$ is a closed contact $3$-manifold supported by a spinal open book decomposition $(P, \phi, \{\Sigma_i\}_{i=1}^m)$. Let $k_i$, $\pi_i$, $n_i$, $B$ and $\rho$ be as above, and suppose that at least one $\pi_i$ has no branch points (i.e., $n_i = 0$). Then, for every Hurwitz equivalence class of positive admissible factorizations of $\phi$ with respect to $(B, \rho)$, there exists a PANLF $(W, \omega)$ with regular fiber $P$ and base $B$, such that $(W, \omega)$ is an exact filling of $(M,\xi)$ inducing the given spinal open book decomposition.
\end{theorem}

To prove the theorems, we need to establish some useful facts. We first construct a topological PANLF from a given spinal open book. The value of the following lemma is the explicit structure of the associated PANLF, which allows us to see the symplectic and Stein structures later.

\begin{lemma}\label{lem:constructing_nearlyLF_top}
  Let $(P,\phi,\{\Sigma_i\}_{i=1}^m)$ be a spinal open book of a closed $3$-manifold $M$. Suppose that $k_i$, $\pi_i$, $n_i$, $B$ and $\rho$ are as above. Then for a positive admissible factorization of $\phi$ with respect to $(B,\rho)$, there exists a corresponding PANLF $\Pi\colon W \to B$ such that $\Pi$ induces the spinal open book on its boundary $\partial W = M$. Also, $W$ can be constructed as follows:
 \begin{itemize}
   \item Take the complement of a small neighborhood of a positive multisection in a bordered Lefschetz fibration without a vanishing cycle (i.e., a fiber bundle with monodromy $\phi_{\rho}$),
   \item add 2-handles along vanishing cycles,
   \item if $n_i \neq 0$ for all $i=1,\ldots, m$, attach  spine-removal cobordisms. 
 \end{itemize}
\end{lemma}

\begin{proof} 
    
  We first assume that at least one $n_i = 0$. By Theorem~\ref{thm: bh_multisections}, a positive admissible factorization of $\phi$ with respect to $(B,\rho)$ corresponds to the complement of a positive multisection $\Sigma$ in a Lefschetz fibration $\Pi \colon E \to B$ with regular fiber $P$. Then we obtain PANLF $\Pi\colon W \to B$ by restricting $\Pi$ to the complement $W = E \setminus N(\Sigma)$. 
    
  Now we construct $W$ from scratch. We can assume that up to symplectic deformation, the critical points of $\Pi$ are located in a collar neighborhood $N$ of $\partial B$, and that all the branch points of $\Pi|_{\Sigma}$ are located in $B_0 = B - N$. Now we consider $E_0 := \Pi^{-1}(B_0)$, and the restriction $\Pi \colon E_0 \to B_0$ is a bordered Lefschetz fibration without a critical point. Now we obtain $W$ by taking the complement of the multisection $\Sigma \cap E_0$ in $E_0$ and then attaching $2$-handles along the vanishing cycles corresponding to the critical points on $N$.  
     
  Next, assume $n_i > 0$ for all $i$. Then consider another spinal open book $(P',\phi',\{\Sigma_i\}_{i=0}^m)$, where $P' = P \setminus \bigsqcup_{|\partial B|}\mathbb{D}^2$, $\phi'$ is a lift of $\phi$ on $P'$ and $\Sigma_0 = B$ such that $\pi_0 =\text{id}$ and $n_0=0$. Since $n_0 = 0$, there exists a PANLF $\Pi' \colon W' \to B$ corresponding to the spinal open book $(P',\phi',\{\Sigma_i\}_{i=0}^m)$. Let $C$ be a spine removal cobordism from $(P',\phi',\{\Sigma_i\}_{i=0}^m)$ to $(P,\phi,\{\Sigma_i\}_{i=1}^m)$. According to Theorem~\ref{thm: spineremoval}, we know $C = \mathbb{D}^2 \times B$ and we can extend $\Pi'$ to $\Pi \colon W = W' \cup C \to B$ and it is a desired PANLF corresponding to the spinal open book $(P,\phi,\{\Sigma_i\}_{i=1}^m)$.  
\end{proof}

We now claim that the PANLF we constructed in Lemma~\ref{lem:constructing_nearlyLF_top} admits a symplectic structure. To show this, we make use of the following observation, which is a special case of \cite[Theorem 1.2]{hayden2021quasipositive} and also follows from work of Baykur-Hayano \cite{baykur2016multisections}. 

\begin{observation}\label{obs: qpsurface}
  Any link in the boundary of a positive allowable bordered Lefschetz fibration with planar fibers and without vanishing cycles bounds a positive symplectic multisection in the bordered Lefschetz fibration if it is a quasipositive braid with respect to the induced positive allowable planar spinal open book on the boundary.
\end{observation}

We slightly modify \cite[Theorem 1.2]{marktosun23} to obtain a symplectic structure on the complement of a multisection in a symplectic bordered Lefschetz fibration, and a spinal open book on the boundary. 

\begin{proposition}\label{prop: marktosun}
  Let $(M, \xi)$ a closed contact $3$-manifold supported by a Lefschetz-amenable spinal open book decomposition $(P, \phi, \{\Sigma_i\}_{i=1}^m)$, and $(W, \omega)$ a Stein filling of $(M,\xi)$ that is a bordered Lefschetz fibration inducing the spinal open book on the boundary. Suppose $K$ is a transverse link in $(M, \xi)$ which is an $n$-component quasipositive braid with respect to the spinal open book decomposition. Let $\Delta$ be a properly embedded symplectic surface in $W$ with boundary $K$. Then $\Delta$ has an arbitrarily small tubular neighborhood $U$ such that for $W' := W - U$ and $\omega' := \omega|_{W'}$, the complement  $(W', \omega')$ is an exact symplectic filling of $(M',\xi')$ where $M' = \partial W'$. Moreover, $(M', \xi')$ is supported by a planar spinal open book $(P',\phi',\{\Sigma_i\}_{i=1}^m)$ where $P' = P - \sqcup_n \mathbb{D}^2$, $\phi'$ is a lift of $\phi$ to $P'$, and $\Sigma_0 = \Delta$.  
\end{proposition}

\begin{proof}
  The proof follows almost verbatim the proof of Theorem 1.2 in \cite{marktosun23}. For the reader's benefit we indicate where the situation is slightly different in our case. We consider a neighborhood of the symplectic surface $\Delta$, call it $\Delta_1 \times \Delta_2$. The smoothing for the lower boundary $H$ is done similarly as in \cite{marktosun23}, except that the functions $r_1(\sigma)$ and $r_2(\sigma)$ are defined near the relevant boundary components. In \cite{marktosun23}, the lower boundary gave rise to a new binding component, whereas here it gives rise to a new spine component, with vertebra $\Delta$. Since the original filling $(W, \omega)$ was a positive allowable Lefschetz fibration, hence Stein (by Theorem~\ref{thm: blf_stein}) and in particular exact, the manifold $W'$ is thus an exact filling of $M'$, by an analogous argument to Theorem 3.10 of \cite{marktosun23}.
\end{proof}

We will also establish the following technical fact which is similar to what happens for regular open books.

\begin{lemma}\label{lem: weinstein2handles}
  Let $(M,\xi)$ be a closed contact $3$-manifold supported by a spinal open book $(P, \phi, \{\Sigma_i\}_{i=1}^m)$. Suppose $(M', \xi')$ is another contact $3$-manifold supported by the spinal open book  $(P, \tau_{\gamma} \circ \phi, \{\Sigma_i\}_{i=1}^m)$, where $\gamma$ is a homologically essential curve on $P$. Then there exists a Legendrian representative of $\gamma$ such that the Weinstein $2$-handle cobordism along $\gamma$ sends $(M,\xi)$ to $(M',\xi')$.
\end{lemma}

\begin{proof}
  According to \cite[Lemma~3.2]{etnyre_shea}, the page $P$ can be made convex such that the dividing set consists of boundary parallel arcs. We can isotope $\gamma$ on $P$ to avoid the dividing set. Then, by applying the Legendrian realization principle, we can assume $\gamma$ is Legendrian with the contact framing equal to the framing coming from $P$. Since a positive Dehn twist about $\gamma$ corresponds to a $(-1)$-surgery with respect to the surface framing, it corresponds to a Legendrian surgery. Therefore, a Weinstein 2-handle attachment corresponding to the Legendrian surgery produces a Weinstein cobordism from $(M, \xi)$ to $(M',\xi')$.
\end{proof}

We are now ready to prove Theorems~\ref{thm: factorization_filling} and \ref{thm: exact_filling}.

\begin{proof}[Proof of Theorems~\ref{thm: factorization_filling} and \ref{thm: exact_filling}]
  By Lemma~\ref{lem:constructing_nearlyLF_top}, every positive admissible factorization of $\phi$ with respect to $B$ corresponds to a PANLF $(E\to B)$ with boundary $M$, obtained by removing a multisection $\Sigma$ from a symplectic Lefschetz fibration $(W,\omega)$. Further, according to the proof of Lemma~\ref{lem:constructing_nearlyLF_top}, we can construct $E$ as follows:
  \begin{enumerate}
    \item[Step 1.] Take the complement of a small neighborhood of a multisection in a bordered Lefschetz fibration with no vanishing cycles (i.e., a fiber bundle with monodromy $\phi_{\rho}$),
    \item[Step 2.] add 2-handles along vanishing cycles,
    \item[Step 3.] attach  spine-removal cobordisms if $n_i \neq 0$ for all $i=1,\ldots, m$.
  \end{enumerate}
  By Theorem~\ref{thm: blf_stein}, a bordered Lefschetz fibration without a vanishing cycle admits a Stein structure. According to Hayden \cite{hayden2021quasipositive}, the boundary of a positive multisection in a Stein surface is a quasipositive braid. Therefore, we can apply Proposition~\ref{prop: marktosun} and the 4-manifold constructed in Step 1 admits a symplectic structure which is an exact filling of the boundary. The 2-handles attached along vanishing cycles admits a Weinstein structure by Lemma~\ref{lem: weinstein2handles}. A Weinstein cobordism preserves the exact fillability and this establishes Theorem~\ref{thm: exact_filling}. 
    
  In case of $n_i \neq 0$ for all $i = 1,\ldots,m$, we need to attach spine removal cobordisms in Step 3 as shown in the proof of Lemma~\ref{lem:constructing_nearlyLF_top}. A spine removal cobordism is a weak symplectic cobordism. However, since the symplectic manifold constructed in Step 2 is exact, so the spines remain exact even after attaching the spine removal cobordisms. Hence by \cite[Theorem 1.10]{spinal_II}, the filling can be deformed into a strong filling and this completes the proof of Theorem~\ref{thm: factorization_filling}.
\end{proof}

We further show that for a specific class of planar spinal open books, we can deform the minimal strong fillings constructed in Theorem~\ref{thm: factorization_filling} into Stein fillings. The idea is that in the standard $\bD^4$ we can deform a positive multisection into a complex surface, and the complement of the complex surface admits a Stein structure.

In addition to Proposition \ref{prop: marktosun}, we use the following classical result to obtain a Stein structure in the complement of a complex hypersurface. We omit the definition of a {\bf normal Stein space} -- the relevant fact is that the standard $\R^4$ seen as a completion of the standard $4$-ball is a normal Stein space. We recall the following fact.
%p.130
\begin{theorem} \label{thm: complement_stein}\cite[p.130]{Theory_of_Stein_spaces}\cite[Theorem 1]{simha}
  If $X$ is a normal Stein space of complex dimension two and $H$ is an analytic hypersurface in $X$, then $X\backslash H$ is Stein.
\end{theorem}

Notice that in the trivial Lefschetz fibration $D^4 = D^2 \times D^2$, a positive multisection of the base $D^2$ is equivalent to the quasipositive surface; see \cite[Definition 1.1]{boileau-orevkov}. Thus we can restate \cite[Theorem 2]{boileau-orevkov} as follows:

\begin{theorem}\label{thm: sympl_complex}\cite{boileau-orevkov}
  Any symplectic surface in $(\bD^4, \omega_{st})$ is symplectically isotopic to a positive multisection in the Lefschetz fibration $\mathbb{D}^2\times \mathbb{D}^2 \to \mathbb{D}^2$.
  %, and it is further symplectically isotopic to a complex curve.
\end{theorem}

By applying Theorem~\ref{thm: complement_stein} and Theorem~\ref{thm: sympl_complex}, we can improve Theorem~\ref{thm: exact_filling} in the case of $B =\mathbb{D}^2$ and obtain a Stein filling of the given spinal open book.

\begin{theorem}[Theorem \ref{thm: panlf_construct_stein}]\label{thm: stein_filling}
  Suppose $(M,\xi)$ is a closed contact $3$-manifold supported by a planar spinal open book decomposition $(P, \phi, \{\Sigma_i\}_{i=1}^m)$ where $\Sigma_1 = \mathbb{D}^2$. Let $k_i$, $\pi_i$, $n_i$, $B$ and $\rho$ be as in Theorem~\ref{thm: factorization_filling}. Assume further that $\pi_1$ has no branch points, i.e., $n_1 = 0$. Then, for every Hurwitz equivalence class of positive admissible factorizations of $\phi$ with respect to $(B, \rho)$, there exists a symplectic PANLF $(W, \omega)$ with regular fiber $P$ and base $B$, such that $(W, \omega)$ is a Stein filling of $(M,\xi)$ inducing the given spinal open book decomposition.
\end{theorem}

\begin{proof}
  By Lemma \ref{lem:constructing_nearlyLF_top} and its proof, $(M,\xi)$ has the spinal open book induced by a filling $E$ by PANLF which is constructed by the following. First, consider the complement of a multisection $S$ in $\natural_n S^1\times \bD^3$, with boundary a quasipositive link $K$ with respect to a planar open book supporting $(\#_n S^1 \times S^2, \xi_{st})$, for some $n$. Here $S$ is determined by $\{\Sigma_i\}_{n_i \neq 0}$, since we know $B= \bD^2$. Second, we add $2$-handles along vanishing cycles corresponding to $\phi$ to obtain $E$.

  Now in order to see that this procedure in fact gives rise to a Stein filling, we need to examine the multisection complement $E_1 := \natural_n S^1\times \bD^3 \setminus S$ in the first step. Observe that $E_1$ can also be obtained by removing a symplectic multisection $S'\cup \sqcup_{n} D^2$ from the $\bD^2 \times \bD^2$ fibration of $\bD^4$, where $S'$ is homeomorphic to $S$. The boundary of $S' \cup \sqcup_{n} \bD^2$ is a quasipositive link $K'$ in the disk open book obtained by considering the braid $K$ in the planar open book supporting $(\#_n S^1 \times S^2, \xi_{st})$, and the boundaries of the disks correspond to all but one boundary components in the same planar open book -- these give trivial components in the link $K'$. 

  By Theorem \ref{thm: sympl_complex}, we know that $S'\cup \sqcup_{n} \bD^2$ is symplectically isotopic to a positive multisection in $\bD^2 \times \bD^2$, denoted as $\mathcal{S'}$. Furthermore, by \cite{rudolph2004algebraic} we know that $\mathcal{S}'$ is smoothly isotopic to a complex curve, denoted as $\mathcal{S}$. Applying Theorem \ref{thm: complement_stein}, we obtain that $E_1$, the complement of $\mathcal{S}$ in $\bD^4$, is Stein. 
  %%%%%%%%%%more detailed description of psh%%%%%%%%%%%%%
  %Note that by the proof of Lemma 1 in \cite{simha}, we may pick the new plurisubharmonic function given in Theorem \ref{thm: complement_stein} to have a level set whose contact structure agrees with the boundary of the complement of the neighborhood of $\mathcal{S}$ in $\bD^4$, which by construction is the contact structure supported by the given spinal open book decomposition that we desire. 
  In fact, we may pick an explicit plurisubharmonic function on $\bD^2 \times \bD^2 \setminus \mathcal{S}$ to be $f:=\max(-\log(d_{\cS}), |z_1|^2 + |z_2|^2)$ where $d_{\mathcal{S}}$ denotes the distance function to the set $\mathcal{S}$ and $z_1, z_2$ denote the two complex coordinates on $\bD^2 \times \bD^2$. 
  
  We now check that a level set of $f$ is compatible with a contact structure supported by the spinal open book obtained by removing $\mathcal{S}$. Notice that the induced contact form on a regular level set of $f$ is $\lambda := -df \circ J_0$ where $J_0$ is the standard complex structure on $\bD^4$. We first see that $d\lambda$ is a symplectic form on each regular fiber and hence each page of the open book. Second, we claim that $\lambda$ evaluates positively on each fiber of the spine. Notice that removing a small neighborhood of $\mathcal{S}$ introduces some boundary components on each fiber surface $\{p\} \times \bD^2$ for $p \in \bD^2$. The $S^1$-fibration on the spine coincides with the boundaries of the fibers away from the exotic curves, and near the exotic curves can be made arbitrarily close to the boundaries of the fiber. Thus to show the claim, it is sufficient to show that $\lambda$ evaluates positively on tangent vectors of each boundary component of each fiber surface. Let $X$ be a non-vanishing tangent vector field along a boundary component of a fiber surface, inducing the boundary orientation. Then $J_0(X)$ points away from the the multisection $\mathcal{S}$, so the function $f$ decreases along the integral curves of $J_0(X)$. Therefore, it follows that $\lambda(X)>0$ and this completes the proof that $\ker(\lambda)$ is supported by the induced spinal open book on the boundary.
  
  Finally, as $E$ is obtained by attaching Weinstein handles to $E_1$, corresponding to vanishing cycles, $E$ is Stein as well and the boundary contact structure is supported by the induced spinal open book. 
\end{proof}

%%%%%%%%%%%%%%%%%%%%%%%%%%%%%%%%%%%%%%%%%%%%%%%%%%%%%%%%%%%%
\section{Classifications of fillings} \label{sec: classification_filling}
%%%%%%%%%%%%%%%%%%%%%%%%%%%%%%%%%%%%%%%%%%%%%%%%%%%%%%%%%%%%

In this section, we apply Theorem \ref{thm: monodromy_fact} to classify symplectic fillings of some contact $3$-manifolds supported by planar spinal open books, and prove Theorem \ref{thm: torusbundle_spinal}, \ref{thm: parabolicbundles}, \ref{thm: ellipticbundles}, \ref{thm:eg1_33}, \ref{thm:eg1_34}, and \ref{thm: highgenus}. 

%------------------------------------------
\subsection{Fillings of planar spinal open books}
%------------------------------------------
We first show that any rotational contact torus bundles with at least $\pi$ twisting are supported by a planar spinal open book.

\begin{proposition}\label{prop: bundle_planar}
  Let $(M,\xi_n)$ be a contact torus bundle over $S^1$ with $n\geq 1$, i.e., rotational contact structure with at least $\pi$ twisting. Then $\xi_n$ is supported by a planar spinal open book with vertebra $\Sigma = (S^1 \times I) \sqcup_{i=1}^k \mathbb{D}^2$, a disjoint union of an annulus and $k$ disks, and page $S^2_{k+2} = S^2 - \sqcup_{k+2} \mathbb{D}^2$, a $(k+2)$-punctured sphere for some $k \in \N$. 
\end{proposition}

\begin{proof}
  Let $(T_{A},\xi_n)$ be a rotational contact torus bundle for $A \in SL_2(\mathbb{Z})$ and $n \geq 1$. As discussed in \S\ref{subsec:relobd}, a rotational torus bundle $(T_{A},\xi_n)$ can be represented by an open book of the form shown in Figure~\ref{fig:vanhornmorris}. Also, by inserting the word $(aba)^{-2}$, we can obtain an open book for $(T_{-A},\xi_{n+1})$, see Figure~\ref{fig:piTwisting} for an example. This open book can be seen as a union of two relative open books $\mathcal{O}_1$ and $\mathcal{O}_2$, where $\mathcal{O}_1$ corresponds to the word $(aba)^{-2}$ and $\mathcal{O}_2$ is the complement of $\mathcal{O}_1$. The page of $\mathcal{O}_2$ is $S^2_{k+2} = S^2 - \sqcup_{k+2} \mathbb{D}^2$ where $k$ is the number of binding components of $\mathcal{O}_2$. Let $\phi$ be the monodromy of $\mathcal{O}_2$.

  According to Proposition~\ref{prop:obd_ann_vertebra}, by replacing $\mathcal{O}_1$ with an annulus vertebra we can construct a spinal open book $(S^2_{k+2},\phi,\{\Sigma\}_{i=0}^k)$ where $\Sigma_i= \mathbb{D}^2$ for $i = 1, \ldots, k$ and $\Sigma_0 = S^1 \times I$. Clearly the spinal open book supports the contact torus bundle $(T_{-A},\xi_{n+1})$, as the annulus spine component is contactomorphic to a relative open book of $T^2 \times I$ corresponding to the word $(aba)^{-2}$.
\end{proof}

\begin{figure}[htbp]{
  \vspace{0.2cm}
  \begin{overpic}[tics=20]
  {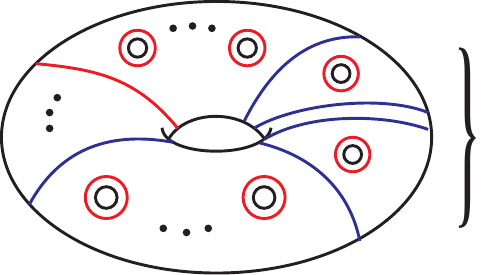}
  \put(238,63){\large$\mathcal{O}_1$}
  \end{overpic}}
  \vspace{0.2cm}
  \caption{An open book decomposition supporting a rotational contact torus bundle $(T_{-A},\xi_{n+1})$. $\mathcal{O}_1$ is a relative open book corresponding to $(aba)^{-2}$. The complement of $\mathcal{O}_1$ is $\mathcal{O}_2$.}
  \label{fig:piTwisting}
\end{figure}

We first classify symplectic fillings of elliptic and parabolic torus bundles with $\pi$ twisting. By Proposition~\ref{prop: bundle_planar}, they are supported by planar spinal open books with connected pages, one annulus spine, and the rest disk spines. Here, classification of fillings amounts to finding factorizations of monodromy that must include interchanging the two boundary components touching the annulus spine twice, and the remaining factors are positive twists about nonseparating simple closed curves. First, we establish the effect on the page boundary of interchanging a pair of boundary components twice along non-intersecting arcs. Since each half-twist happens in a neighborhood of the arc containing the two boundary components, the effect of two half-twists about two non-intersecting arcs is represented on a four-punctured sphere.

\begin{figure}[htbp]{
  \vspace{0.2cm}
  \begin{overpic}[tics=20]
  {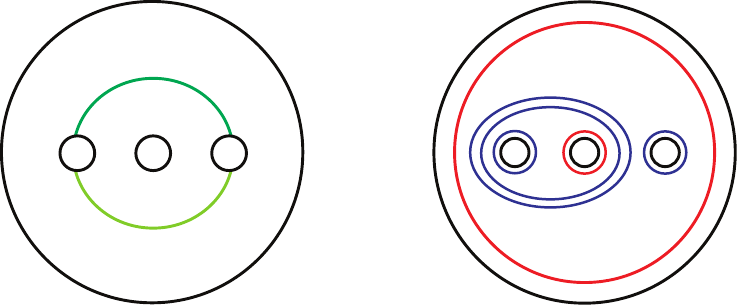}
  \put(70,115){$\alpha$}
  \put(70,24){$\beta$}
  \end{overpic}}
  \vspace{0.2cm}
  \caption{The mapping class is equivalent to doing two consecutive boundary interchanges of the same pair along two non-intersecting arcs in a planar surface as pictured, such that the right boundary component is moved around in a counterclockwise manner. The mapping class pictured is $\tau_{\beta}\tau_{\alpha}$.}
  \label{fig:twotwists}
\end{figure}

\begin{proposition}\label{prop:twotwists}
  Interchange two boundary components via positive half twists along the arcs as shown in the left side of Figure~\ref{fig:twotwists}. Then the resulting mapping class on the surface is equivalent to the one shown in the right side of Figure~\ref{fig:twotwists}.
\end{proposition}

\begin{proof}
  This follows from Figure~\ref{fig:twotwists2}. The idea there is to realize the boundary interchange along $\alpha$, followed by that along $\beta$, as the simultaneous occurrence of the following three events:
  \begin{itemize}
    \item The boundary on the right endpoint of $\alpha$ moving (analogous to moving a marked point) all the way around, counterclockwise, along a loop that encloses the two other interior boundary components
    \item The boundary on the left endpoint of $\alpha$ moving, similarly as above, all the way around, clockwise, along a loop that encloses the innermost boundary component
    \item A full left-handed twist about each of the boundary components at either end of $\alpha$
  \end{itemize}

  The mapping class for moving a marked point along a loop is described in \cite[Fact 4.7]{primer_mcg}.
\end{proof}

\begin{figure}[htbp]{
  \vspace{0.2cm}
  \begin{overpic}[tics=20]
  {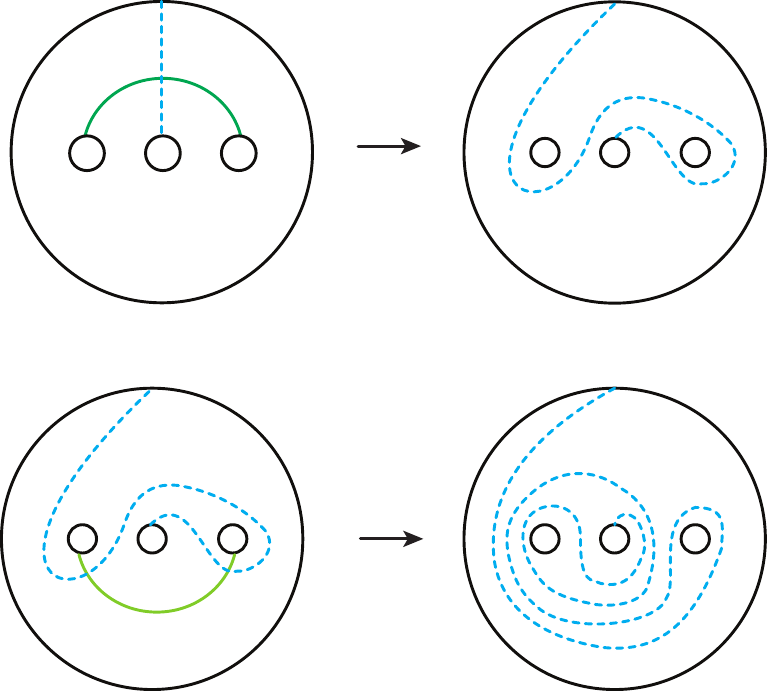}
  \end{overpic}}
  \vspace{0.2cm}
  \caption{The effect of boundary interchange on $\alpha$ followed by that on $\beta$, on the blue arc.}
  \label{fig:twotwists2}
\end{figure}

We also point out that monodromy factorizations that are related by framing conjugations give symplectic fillings that are deformation equivalent according to \cite[p.17]{Baykur_Hayano_Hurwitz}. This allows us to relate monodromy factorizations which have boundary interchanges about arcs that are homotopic, but not homotopic rel boundary.

\begin{lemma}\label{lem:framing_conj}
  Consider two arcs $\alpha$ and $\alpha'$ connecting the boundary components $\partial_1$ and $\partial_2$ as shown in Figure~\ref{fig:arcs}. Then, 
  \[
    \tau_{\alpha'}= t_{\partial_1}t_{\partial_2}^{-1}\tau_{\alpha}
  \]
\end{lemma}

\begin{figure}[htbp]{
  \vspace{0.2cm}
  \begin{overpic}[tics=20]
  {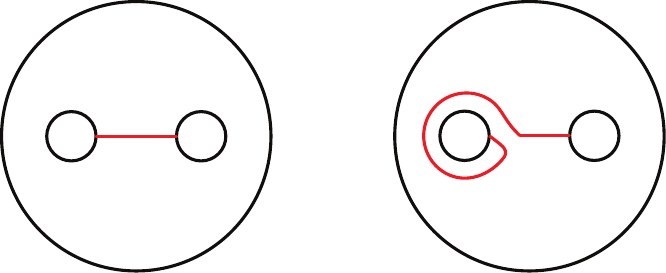}
  \end{overpic}}
  \vspace{0.2cm}
  \caption{Two arcs $\alpha$ and $\alpha'$ joining two boundary components.}
  \label{fig:arcs}
\end{figure}

Now we recall Theorem \ref{thm: parabolicbundles} and prove it.

\parabolic*

\begin{proof}
  By applying Proposition~\ref{prop: bundle_planar}, we obtain a planar spinal open book supporting $(T_-(n), \xi_1)$, which has an annulus page with monodromy being $t_{\gamma}^n$ where $\gamma$ is the core of the annulus, and a single annulus vertebra. The spinal open book is uniform with respect to $\mathbb{D}^2$, and any strong filling contain exactly two exotic fibers by Theorem~\ref{thm: intro-localmodel}. Further, by Theorem~\ref{thm: stein_filling}, any minimal strong filling is deformation equivalent to a Stein filling.
    
  By Theorem \ref{thm: monodromy_fact}, minimal strong fillings of $(T_-(n),\xi_1)$ correspond to Hurwitz equivalence classes of positive admissible factorizations of $t_{\gamma}^n$. Due to the existence of two exotic fibers in the filling, there shuold be two boundary interchanges in the factorization. Let $\alpha$ be an arc connecting two boundary components of the page. Then by Proposition~\ref{prop:twotwists}, we have the monodromy factorization $t_{\gamma}^{-4} = \tau^2_{\alpha}$. From these, it is straightforward to see that there is a unique factorization of $t^n_{\gamma}$ up to Hurwitz equivalence: 
  \[
    t^n_{\gamma} = t_{\gamma}^{n+4}\circ \tau^2_{\alpha}
  \]
  Thus $(T_-(n),\xi_1)$ is not strongly fillable if $n < -4$ and there is a unique Stein filling for $\xi_1$ if $n \geq -4$. 
\end{proof}

Now we prove Corollary~\ref{cor: not_monoid}.

\begin{proof}
  Consider a spinal open book decomposition $(P,\phi_n,\Sigma)$ where $P$ and $\Sigma$ are an annulus, and $\phi_n = t_{\gamma}^{n}$ for $\gamma$ being the core of $P$. Then we have 
  \[
    \phi_{m+n} = \phi_{m} \circ \phi_{n}.
  \] 
  Thus $\phi_{-5}$ is a product of $\phi_{-2}$ and $\phi_{-3}$. According to Theorem~\ref{thm: parabolicbundles}, however, $(P,\phi_n,\Sigma)$ is Stein fillable for $n = -2,-3$, but not strongly fillable for $n = -5$.
\end{proof}

We now prove Theorem~\ref{thm:eg1_33}, by proving the following equivalent statement.

\begin{theorem}
  Let  $(M, \xi)$ be a contact $3$-manifold supported by a spinal open book $(P,\phi,\Sigma)$ where $P$ is an annulus, $\Sigma = \mathbb{D}^2$, and $\phi = \tau_{\alpha}$, a boundary interchanging map along an arc $\alpha$ connecting two boundary components of $P$. Then $(M, \xi)$ bounds a unique strong filling which is deformation equivalent to a Stein filling.
\end{theorem}

\begin{proof}
  Since $(M, \xi)$ is supported by a planar spinal open book, by Theorem~\ref{thm: monodromy_fact}, all minimal strong fillings of $(M, \xi)$ arise from positive admissible factorizations of the page monodromy with respect to a surface $B$, where the spinal open book is uniform with respect to $B$. As the single boundary component of the page component has multiplicity 2, and the vertebra is a disk, $B$ has to be a disk and we consider the two-fold branched cover $\mathbb{D}^2 \to \mathbb{D}^2$ with a single branch point. It follows that any filling of $(M, \xi)$ has the structure of a PANLF with base $\mathbb{D}^2$ and annulus fiber. Then, by Theorem~\ref{thm: intro-localmodel}, the completion of any filling will have a single exotic fiber, the monodromy around which will interchange the boundary components of the annulus. It follows that there is a unique such a PANLF, hence $(M, \xi)$ has a unique strong filling up to deformation, which is deformation equivalent to a Stein filling.
\end{proof}

In the following, we prove Theorem~\ref{thm:eg1_34}. As before, we prove a restatement in terms of the planar spinal open book decomposition of the manifold. 

\begin{theorem}
  Consider the contact $3$-manifold $(M, \xi)$ supported by the planar spinal open book with two paper components with annulus fibers and monodromy interchanging boundary components, and a single spine component with $\Sigma_{2,2}$ vertebra. The Stein fillings of this manifold are either bordered Lefschetz fibrations with $\Sigma_{1,2}$ base and annulus fibers with no singular fibers, or a PANLF with annulus base, annulus fibers, and four singular fibers.
\end{theorem}

\begin{proof}
  Since $(M, \xi)$ is supported by a planar spinal open book, by Theorem~\ref{thm: monodromy_fact}, all minimal weak fillings of $(M, \xi)$ arise from positive admissible factorizations of the page monodromy with respect to base $B \in \mathcal{B}$. Here, $B$ has to be a surface such that there exists a two-fold branched cover $\Sigma_{2,2} \to B$ which induces a two-fold honest cover on the boundaries. From the Riemann--Hurwitz formula, $B$ should be either $\Sigma_{1,2}$ or $\Sigma_{0,2}$. 
    
  In the case $B = \Sigma_{1,2}$, by Riemann--Hurwitz formula, we have an unbranched covering. The fillings of $(M, \xi)$ which have a foliation parameterized by $\Sigma_{1,2}$ thus correspond to positive allowable factorizations of the page monodromy with respect to $\Sigma_{1,2}$, and there are no exotic fibers, hence no boundary interchanges due to those. Thus these fillings are bordered Lefschetz fibrations. Notice that we obtain a family of such fillings by varying the monodromy representation.

  In the case $B = \Sigma_{0,2}$, by Riemann--Hurwitz formula, there are four branch points. By Theorem~\ref{thm: intro-localmodel}, the corresponding fillings are PANLFs with annulus fibers and annulus base, and their completions have 4 exotic fibers. The PANLF with annulus fibers and annulus base, whose completion has $4$ exotic fibers, and no singular fibers, induces a spinal open book on its boundary with two page components with annulus fibers, such that monodromy on one of the page components interchanges boundary components, while on the other page component the monodromy is composition of three boundary interchanges, which is equivalent to one boundary exchange composed with $\tau_{\gamma}^{-4}$ by Proposition~\ref{prop:twotwists}. Notice that the monodromy representation must be trivial in this case. To agree with the spinal open book monodromy as in the statement of the theorem, the PANLF filling must also contain 4 singular fibers, each of which has $\gamma$ as its vanishing cycle. This gives the unique minimal filling of $(M, \xi)$ that contains exotic fibers.
    
  In fact, the Stein filling $E_2$ constructed in \cite{spinal_II} (see (\ref{eq:E})) contains eight Lagrangian spheres. Since there is a unique PANLF filling of $(M, \xi)$, this Stein filling must be deformation equivalent to the PANLF obtained above.
\end{proof}

\begin{figure}[htbp]{
  \vspace{0.2cm}
  \begin{overpic}[tics=20]
  {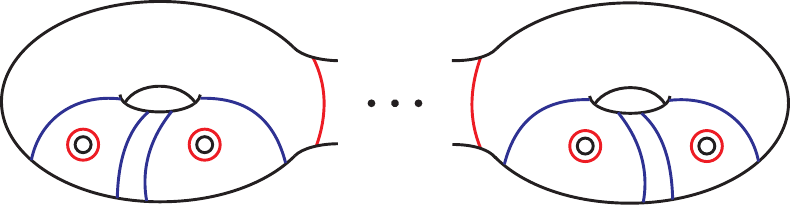}
  \end{overpic}}
  \vspace{0.2cm}
  \caption{A high genus open book which can be turned into a planar spinal open book. }
  \label{fig:highgenus}
\end{figure}

Now we prove Theorem~\ref{thm: ellipticbundles}. Note that the uniqueness of these fillings up to diffeomorphism was proven in \cite{GL:fillings}.

\ellipticbundles*

\begin{figure}[htbp]{
  \vspace{0.2cm}
  \begin{overpic}[tics=20]
  {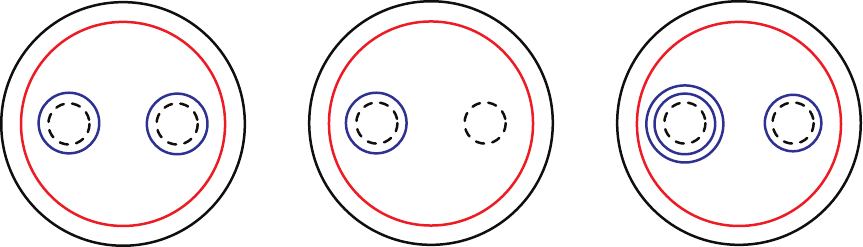}
  \put(55,-12){$\phi_1$}
  \put(205,-12){$\phi_2$}
  \put(355,-12){$\phi_3$}
  \end{overpic}}
  \vspace{0.2cm}
  \caption{The monodromies of the spinal open books supporting the elliptic bundles with $\pi$-twisting. The dotted boundary components meet the spine component with annulus vertebra.}
  \label{fig:ellipticbundles}
\end{figure}

\begin{proof}
  According to \cite[Table in p.6]{honda2} and \cite[Proof of Theorem 4.3.1]{Van_horn_morris_thesis}, there exist three contact elliptic torus bundles with $\pi$ twisting. They are supported by the open books shown in Figure~\ref{fig:piTwisting} corresponding to the words $(aba)^{-2}$, $(ba)^{-5}$, and $(ba)^{-4}$, respectively. By applying Proposition~\ref{prop:obd_ann_vertebra}, we can convert them into planar spinal open books $(P,\phi_i,\Sigma_1\cup\Sigma_2)$ where $P$ is a pair of pants, $\Sigma_1 = \mathbb{D}^2$ and $\Sigma_2$ is an annulus. The monodromies $\phi_i$ are shown in Figure~\ref{fig:ellipticbundles}. The spinal open book is uniform with respect to $\mathbb{D}^2$, and any strong filling contains exactly two exotic fibers by Theorem~\ref{thm: intro-localmodel}. Further, by Theorem~\ref{thm: stein_filling}, any minimal strong filling is deformation equivalent to a Stein filling. By Theorem \ref{thm: monodromy_fact}, minimal strong fillings of the torus bundles correspond to Hurwitz equivalence classes of positive admissible factorizations of $\phi_i$. Due to the existence of two exotic fibers in the filling, there should be two boundary interchanges in the factorization. 
    
  Let $\alpha$ be an arc connecting two boundary components of $P$ that meet the annulus vertebra. Suppose $d_1$ and $d_2$ are curves parallel to the boundary components of $P$ that meet the annulus vertebra, and $d_3$ is a curve parallel to the boundary component of $P$ that meets the disk vertebra. Then by Proposition~\ref{prop:twotwists}, we have the monodromy factorization 
  \[
    \tau^2_{\alpha} = t_{d_1}^{-3} t_{d_2}^{-1} t_{d_3}.
  \] 
  From this, it is straightforward to see that there is a unique factorization for each $\phi_i$ up to Hurwitz equivalence: 
  \begin{align*}
    \phi_1 &= t_{d_1}^2\tau^2_{\alpha},\\
    \phi_2 &= t_{d_1}^2t_{d_2}\tau^2_{\alpha},\\
    \phi_3 &= t_{d_1}\tau^2_{\alpha}.
  \end{align*}    
  Thus each torus bundle admits a unique Stein filling. 
\end{proof}

%------------------------------------
\subsection{Higher genus open books}
%------------------------------------
We conclude our paper providing a motivating family of examples that could lead to the classification of symplectic fillings of open books with higher genus pages in the future.

\begin{theorem}\label{thm:highgenus}
  Consider the open book shown in Figure~\ref{fig:highgenus}. For every genus $g\geq 2$, this represents a Stein fillable contact $3$-manifold that admits at least one Stein filling which is a Lefschetz fibration over the disk with genus $g$ fibers, and is also a nearly Lefschetz fibration over the disk with planar fibers.
\end{theorem}

\begin{proof}
  The proof of Proposition 3.2 in \cite{ding2018strong} shows that this open book monodromy admits a positive factorization, which corresponds to a Stein filling that admits a Lefschetz fibration with genus $g$ fibers. Now, we use Proposition~\ref{prop:obd_ann_vertebra} to replace the monodromy parts on each ``handle'' by an annulus vertebra. Thus this manifold is supported by a planar spinal open book which is Stein fillable and is uniform with respect to $\mathbb{D}^2$. Theorem~\ref{thm: monodromy_fact} then shows that the Stein fillings must also have the structure of a PANLF with planar pages.
\end{proof}

\begin{remark}\label{rem: high_genus_difficulty}
  The above examples, for $g \geq 2$, also show the difficulty of classifying fillings using positive admissible monodromy factorizations, which arise from different choices of arcs on the page along which the boundary interchanges can happen. We will make some comments on the symplectic fillings of the $g=2$ example; let us refer to this contact manifold as $M_2$. First observe that the planar spinal open books are uniform with respect to $\mathbb{D}^2$; there are $g$ vertebra components each of which contribute two branch points, hence any filling will contain $2g$ exotic fibers. 

  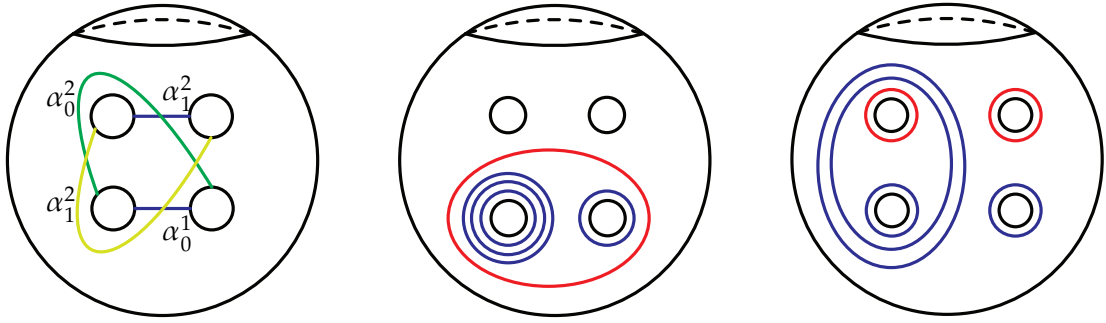
\begin{figure}[htbp]{
    \vspace{0.2cm}
    \begin{overpic}[tics=20]
    {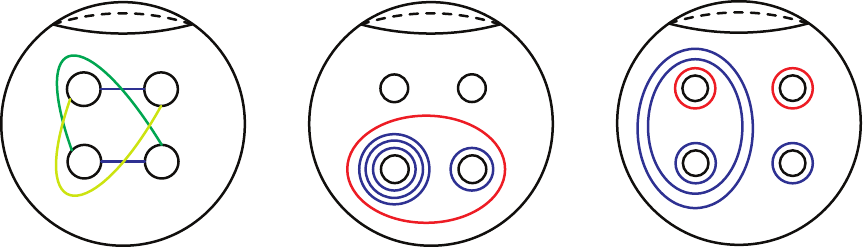}
    \put(15,80){$\alpha_0^2$}
    \put(15,40){$\alpha_1^2$}
    \put(59,29){$\alpha_0^1$}
    \put(59,82){$\alpha_1^2$}
    \end{overpic}}
    \vspace{0.2cm}
    \caption{The arcs $\alpha_1^1$, $\alpha_0^1$, $\alpha_0^2$, and $\alpha_1^2$ on $\Sigma_{0,4}$, connecting pairs of boundary components corresponding to the same spine component in $M_2$. The middle figure is the mapping class $\tau^2_{\alpha_0^1}$, and the rightmost picture represents the mapping class $\tau_{\alpha_0^1}\tau_{\alpha_0^2}$.}
    \label{fig:twoarcs_genus2}
  \end{figure}

  The page on the planar spinal open book is $\Sigma_{0,4}$, a 4-holed sphere, where the boundary components are paired up according to which spine component they meet. Each pair of boundary components meet a spine component with annulus vertebra. The monodromy on the page is $t_{\gamma_{1,2}}^2$ as shown in the bottom right picture in Figure~\ref{fig:genus2_fact2}. In the figure the bottom pair of boundary components meet a spine component, and the top pair meet the other spine component. Since any filling must contain four exotic fibers, two corresponding to each vertebra component, both pairs of boundary components will be interchanged twice in any positive admissible factorization of the monodromy. In Figure~\ref{fig:twoarcs_genus2}, we show two choices of arcs for each pair -- this is not an exhaustive list and part of what makes this a difficult problem! We show that there exist at least three distinct minimal strong fillings of $M_2$, coming from boundary interchanges along different choices of arcs.

  The notation for this example is set up in the following manner:
  \begin{itemize}
    \item The arcs are denoted $\alpha_i^j$, where $i = 0,1$, $j=1,2$, the letter $i$ indexing the lower pair of boundary components as 0 and the upper pair as 1, while the letter $i$ indexes the choice of arc as indicated in Figure~\ref{fig:genus2_fact3}
    \item The boundary interchange mapping class along the arc $\alpha_i^j$ is denoted $\tau_{\alpha_i^j}$
    \item The boundary components are denoted $d_{i}$, the indices counting them clockwise starting from top left as in Figure~\ref{fig:genus2_fact2}
    \item The curve enclosing boundary components $d_i$ and $d_j$ will be denoted $\gamma_{i,j}$.
  \end{itemize}    
  One filling comes from the positive admissible factorization 
  \[
    t_{\gamma_{1,2}}^2 = t_{d_1}^3t_{d_2}t_{d_3}^3t_{d_4}\tau_{\alpha_1^1}^2\tau_{\alpha_0^1}^2
  \]
  which contains 8 vanishing cycles. 

  \begin{figure}[htbp]{
   \vspace{0.2cm}
   \begin{overpic}[tics=20]
   {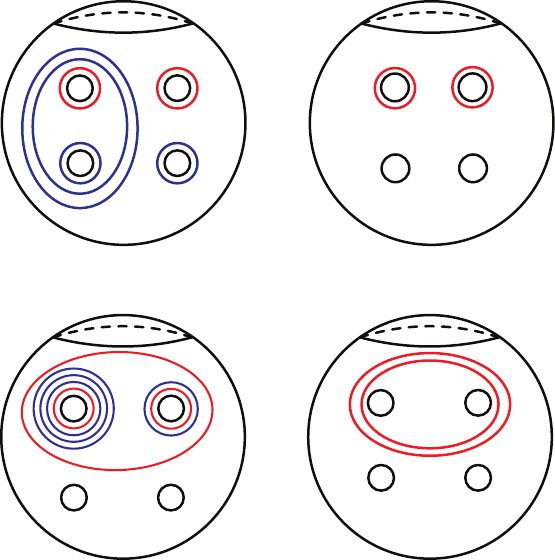}
    \put(45,140){$\tau_{\alpha_0^1}\tau_{\alpha_0^2}$}
    \put(172,140){$t_{d_3}t_{d_4}t^2_{\gamma_{1,4}}\tau_{\alpha_0^1}\tau_{\alpha_0^2}$}
    \put(15,-10){$\tau^2_{\alpha_1^1}t_{d_3}t_{d_4}t^2_{\gamma_{1,4}}\tau_{\alpha_0^1}\tau_{\alpha_0^2}$}
    \put(145,-10){$t^2_{d_1}t_{\gamma_{1,2}}\tau^2_{\alpha_1^1}t_{d_3}t_{d_4}t^2_{\gamma_{1,4}}\tau_{\alpha_0^1}\tau_{\alpha_0^2}$}
   \end{overpic}}
   \vspace{0.5cm}
   \caption{A factorization corresponding to a filling of $M_2$ with 7 singular fibers.} 
   \label{fig:genus2_fact2}
  \end{figure}

  \begin{figure}[htbp]{
   \vspace{0.2cm}
   \begin{overpic}[tics=20]
   {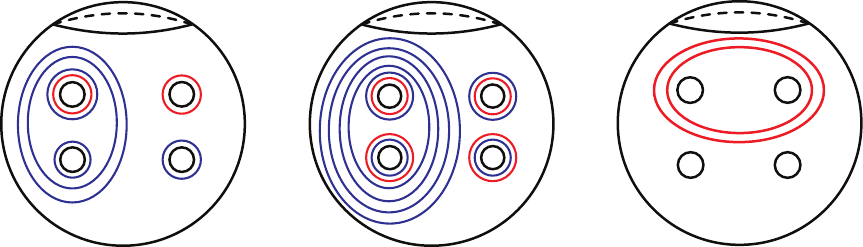}
    \put(40,-10){$\tau_{\alpha_0^1}\tau_{\alpha_0^2}$}
    \put(170,-10){$\tau_{\alpha_1^2}\tau_{\alpha_1^1}\tau_{\alpha_0^1}\tau_{\alpha_0^2}$}
    \put(305,-10){$t_{\gamma_{1,4}}^4t_{\gamma_{1,2}}^2\tau_{\alpha_1^2}\tau_{\alpha_1^1}\tau_{\alpha_0^1}\tau_{\alpha_0^2}$}
   \end{overpic}}
   \vspace{0.5cm}
   \caption{A factorization corresponding to a filling of $M_2$ with six singular fibers.}
   \label{fig:genus2_fact3}
  \end{figure}

  Another filling with 7 singular fibers comes from the positive admissible factorization shown in Figure~\ref{fig:genus2_fact2}, which has 7 vanishing cycles and the product of the boundary interchanges $\tau_{\alpha_1^1}^2\tau_{\alpha_0^1}\tau_{\alpha_0^2}$. Yet another filling with 6 singular fibers comes from the positive admissible factorization shown in Figure~\ref{fig:genus2_fact3}, which has 6 vanishing cycles and the product of the boundary interchanges $\tau_{\alpha_1^2}\tau_{\alpha_1^1}\tau_{\alpha_0^1}\tau_{\alpha_0^2}$. 
\end{remark}

\bibliographystyle{amsalpha}
\bibliography{references}
\end{document}